\documentclass[10pt,reqno]{amsart}
\usepackage[a4paper,top=3cm, bottom=3cm,left=3cm, right=3cm,twoside]{geometry}
\usepackage{amsthm}
 \usepackage{amsmath}
 \usepackage{amsfonts}
 \usepackage{amssymb}
 \usepackage{amscd}
 \usepackage{url}
 \usepackage[utf8x]{inputenc}
\usepackage{csquotes}
 \usepackage[T1]{fontenc}
\usepackage[sfdefault]{FiraSans}
 \usepackage[utf8x]{inputenc}
 \usepackage[english]{babel}
 \usepackage[mathscr]{eucal}
 \usepackage{lmodern}
 \usepackage{bera}
 \usepackage{float} 
 \usepackage{bold-extra}
 \usepackage{textcomp}
 \usepackage{graphicx}
 \usepackage{caption}
\usepackage{subcaption} 
 \usepackage[dvipsnames,table]{xcolor}
 \usepackage{mathtools} 
 \usepackage{enumitem}
 \usepackage{tikz} 
 \usepackage{tikz-cd}
 \usepackage{wasysym}
 \usepackage{todonotes}
 \usepackage{fancybox}
 \usepackage{fancyhdr}
 \usepackage{xspace}
\usepackage[nodisplayskipstretch]{setspace}
\usepackage{hyperref}
\usepackage[maxbibnames=99,
 style=numeric,
    sorting=ynt
]{biblatex}

\DeclareFieldFormat{doi}{}
\DeclareFieldFormat{url}{}
\DeclareFieldFormat{issn}{}
\DeclareFieldFormat{isbn}{}
\DeclareFieldFormat{note}{}

\DeclareDelimFormat{nametitledelim}{\addspace}
\expandafter\let\csname ver@natbib.sty\endcsname\relax

 \theoremstyle{definition}  
  \newtheorem{defn}{Definition}[section]
  \newtheorem{eg}[defn]{Example}

   \newtheorem{rmk}[defn]{Remark}

  \theoremstyle{plain}  
  \newtheorem{thm}[defn]{Theorem}
  \newtheorem{lem}[defn]{Lemma}
  \newtheorem{prop}[defn]{Proposition}
  \newtheorem{cor}[defn]{Corollary}

  \theoremstyle{remark}

 \renewcommand{\it}[1]{\textit{#1}}
 
 \renewcommand{\sf}[1]{\textsf{#1}}

 \newcommand{\mbb}[1]{\mathbb{#1}}
 \newcommand{\mcl}[1]{\mathcal{#1}}

 \newcommand{\mbf}[1]{\mathbf{#1}}
 \newcommand{\msc}[1]{\mathscr{#1}}

 \newcommand{\ol}[1]{\overline{#1}}
 \newcommand{\ul}[1]{\underline{#1}}

 \newcommand{\abs}[1]{\left\lvert#1\right\rvert}

 \newcommand{\norm}[1]{\left\lVert#1\right\rVert}

 \newcommand{\M}[1]{\mbb{M}_{#1}}

 \newcommand{\B}[1]{\msc{B}({#1})}

 \newcommand{\ranko}[2]{|{#1}\rangle\langle{#2}|}

 \newcommand{\ip}[1]{\langle#1\rangle}

 \renewcommand{\ker}[1]{\sf{ker}(#1)}

 \newcommand{\Matrix}[1]{\begin{bmatrix}#1\end{bmatrix}}

 \DeclareMathOperator{\T}{\sf{T}}
 \DeclareMathOperator{\tr}{\sf{tr}}
 
 \DeclareMathOperator{\id}{\sf{id}}

 \DeclareMathOperator{\lspan}{\sf{span}}

 \numberwithin{equation}{section}
 \allowdisplaybreaks[4] 
 \setlist[enumerate]{font=\upshape,noitemsep, topsep=0pt} 
 \setlist[itemize]{noitemsep, topsep=0pt}

\makeatletter
\@namedef{subjclassname@2020}{\textup{2020} Mathematics Subject Classification}
\makeatother

\title{On regions of mixed unitarity for semigroups of unital quantum channels}
\author{B V Rajarama Bhat}
\address{B V Rajarama Bhat\\ Statistics and Mathematics Unit\\ Indian Statistical Institute Bangalore 560059\\ India}
\email{bhat@isibang.ac.in, bvrajaramabhat@gmail.com}

\author{Repana Devendra$^*$}
\address{$^*$ Corresponding author: Repana Devendra\\ Department of Mathematics \\Indian Institute of Technology Bombay \\ Mumbai 400076\\India}
\email{r.deva1992@gmail.com, rdevamaths@gmail.com}

\begin{document}


\begin{abstract}
It is established that both discrete and continuous semigroups of unital quantum channels are {\em eventually mixed unitary.\/} 
This result is novel even for the subclass of  Schur maps and stands in sharp contrast to the resolution of the asymptotic quantum Birkhoff conjecture by Haagerup and Musat, who demonstrated that tensor powers of some unital quantum channels maintain a persistent positive distance from the set of  mixed unitary channels. Remarkably, our results show that this gap vanishes in finite time when considering ordinary powers within a semigroup. 

Building on this, we define the mixed unitary index of a unital quantum channel as the minimum time (or power)  beyond which all subsequent maps become mixed unitary. We demonstrate that for any fixed dimension $d \geq 3$, there is no universal upper bound for this index. Furthermore, we  observe that if a continuous semigroup is not mixed unitary at some $t > 0$, it remains non-mixed unitary for all $t$ sufficiently close to the origin. Finally, we investigate quantum dynamical semigroups where mixed unitarity is restricted to specific families, such as Weyl or diagonal unitaries. We show that Schur semigroups of correlation matrices eventually become mixtures of rank-one correlation matrices, and we characterize the generators of Schur semigroups that remain within this set for all $t \geq 0$.
\end{abstract}



\maketitle

\section{Introduction}
In quantum information theory, a quantum channel (or simply a channel) serves as a fundamental mechanism for transmitting quantum information. Formally, it is defined as a trace-preserving, completely positive linear map from one matrix algebra to another.

For a natural number $d$, let $\M{d}$ denote the algebra of all $d\times d$  complex matrices. In this article, we will be restricting ourselves to studying unital quantum channels (doubly stochastic maps) from $\M{d}$ to $\M{d}$. A \it{unitary quantum channel} on $\M{d}$ is a quantum channel which is given by $X\mapsto U^*XU$, for all $X\in\M{d}$ and for some unitary $U\in\M{d}$. Unitary quantum channels are  extreme in the convex set of unital quantum channels on $\M{d}$. A quantum channel is said to be a \it{mixed unitary} if it is in the convex hull of the set of all unitary quantum channels, or equivalently (due to finite dimensionality) it is a convex combination of unitary channels.  That is, a quantum channel $\Phi:\M{d}\to\M{d}$  is  a mixed unitary if there exist  $\{\lambda_j\}_{j=1}^{\ell}\subseteq [0,1]$  and unitaries $\{U_j\}_{j=1}^\ell\subseteq\M{d}$ for some $l\in \mathbb{N}$ such that $\sum_{j=1}^\ell\lambda_j=1$ and
 \begin{align}
     \Phi(X)=\sum_{i=1}^{\ell}\lambda_i U_i^{*}X U_i, ~~~~~~~\forall X\in \M{d}.
 \end{align}
The set of all mixed unitary channels on $\M{d}$ forms a compact convex subset of all unital quantum channels.

Mixed unitary channels play a significant role in quantum information theory for several key reasons. They provide a rich class of examples that capture key properties of general quantum operations \cite{Ros08}. By virtue of their structural simplicity, they remain highly useful for theoretical analysis, often proving more tractable than arbitrary quantum channels. These channels emerge naturally in multiple contexts, including majorization frameworks based on Hermitian operators \cite{AU82}, as well as in quantum cryptographic protocols involving state encryption \cite{AMTB00, HDPW04, AS04}.

Taking inspiration from the classical Birkhoff-von Neumann theorem, which realizes all doubly stochastic maps as convex combinations of permutations, one may suspect that perhaps every unital quantum channel is mixed unitary. A well-known result \cite[Proposition 2.6]{BPS93}, \cite[Theorem 1]{LaSt93} establishes this for $d=2$. However, this property fails for $d\geq 3$ as first demonstrated by Tregub \cite{Tre86} and  K\"ummerer-Maassen \cite{KuMa87}. See also \cite{BPS93, LaSt93,MeWolf09, Wat18}. Additionally, determining whether a given unital channel is mixed unitary or not  is known to be NP-hard \cite{LWat20}.  

Fix a dimension $d\in \mathbb{N}$ and consider the {\em completely depolarizing channel} on $\M{d}:$
\begin{equation}\label{depolarizing}\delta _d(X):=\frac{ \mbox{trace}(X)}{d}I_d, ~~X\in \M{d}.\end{equation}
A  famous theorem of Watrous \cite{Wat09},  says that  there is a ball of positive radius around $\delta _d$ such that all unital channels in it are mixed unitary.  As a consequence, any sequence of unital channels converging to the completely depolarizing channel eventually consists of only mixed unitary channels. This was  taken note of in the context of powers of primitive unital channels in Section 4 of  \cite{HRS20}. They  define {\em mixed unitary time (or index)\/} of a unital quantum channel as  a threshold $k$ such that all subsequent powers beyond which are mixed unitary. Moreover they  show that such a threshold exists for primitive unital channels and  obtain an upper bound for it. 

In a more recent work, Kribs et al., \cite{KJRM23} extend this idea very significantly.  Note that the completely depolarizing channel can be viewed as the unique trace-preserving conditional expectation map from $\M{d}$ to the $C^*$-algebra $\mathbb{C}I_d$ of scalars. Fix a unital $C^*$-subalgebra $\mcl{A}$ of $\M{d}$ and let $\mbf{E}_{\mcl{A}}$ be the unique trace-preserving conditional expectation to it. It is a unital quantum channel. In \cite{KJRM23} it is observed that unital quantum channels sufficiently close to $\mbf{E}_{\mcl{A}}$  and fixing elements of $\mcl{A}$ are mixed unitary. Armed with this result, they show that if $\Phi $ is a unital quantum channel on $\M{d}$, then there exists some $k\in \mathbb{N}$ such that $\Phi ^k$ is mixed unitary \cite[Theorem 22]{KJRM23}. However,  $\Phi^k $ fixes elements of an algebra $\mcl{A}$, does not imply that $\Phi ^{k+1}$ does the same. For this reason, they couldn't show that $\{\Phi ^n\}_{n\geq 0}$ has finite mixed unitary time. We overcome this difficulty in Theorem \ref{thm-even-mu-discrete}.    In \cite{HaaMu11}, Haagerup and Musat showed that in general tensor powers of unital quantum channels can maintain a positive distance from the set of mixed unitary channels. This settled the asymptotic quantum Birkhoff conjecture in the negative. In striking difference to this, here we see that ordinary powers of unital quantum channels become mixed unitary in finite time! This raises the natural question as to for fixed $d$, whether there exists a bound $N_d$ such that
mixed unitary indices of all unital quantum channels on $\M{d}$ are bounded by $N_d$. We answer this question in the negative in Theorem \ref{thm-No-univ-MUindex}. We get it as a consequence of our analysis of continuous (one parameter)  semigroups of unital quantum channels.


A deep study of continuous semigroups of mixed unitary channels was carried out long time ago by K\"ummerer and Maassen in \cite{KuMa87}. Their goal was to understand  Markov dilations of such semigroups. In the process they characterized generators of quantum dynamical semigroups which remain mixed unitary for all times. We obtain a different characterization of such generators (Theorem \ref{new characterization}). We show that even in continuous time, semigroups of unital channels eventually become mixed unitary. In other words, they all have finite mixed unitary index (Theorem \ref{thm-semigp-eve-MU}).  We may note that there are no classical analogues for this kind of statements as classically all doubly stochastic matrices are convex combinations of permutation matrices.

As a special case of results above we see that  if $[c_{jk}]\in\M{d}$ is a correlation matrix (positive semidefinite with diagonal entries equal to one), its Schur powers $\{[c_{jk}^n]: n\geq 0\}$  are eventually mixed rank-one matrices, that is, they are convex combinations of complex rank one correlation matrices and similar result holds for continuous Schur semigroups (Theorem \ref{thm-schur-even}).  Furthermore, we provide a complete characterization of the generators $A=[a_{jk}]$, for which the  correlation matrices $[e^{ta_{jk}}]$ are mixed rank-one for all times $t\geq 0$ (See Theorem \ref{thm-schur-gene}). 

We organize the paper as follows: In Section 2, we recall the necessary definitions and basic results required in the subsequent sections, and fix some notations. Section 3, investigates the mixed unitarity for a discrete semigroup $\{\Phi^n\}_{n\in\mbb{N}}$ generated by an unital quantum channel $\Phi$ on $\M{d}$. We first establish  that every trace-preserving $*$-automorphism on a $C^*$-subalgebra $\mcl{A}$ of $\M{d}$ is necessarily of the form $\mathrm{Ad}_U$ for some $U\in\M{d}$ (See Lemma \ref{thm-stru-automor-finite}). This helps us to prove that any unital quantum channel on $\M{d}$ is asymptotically  a conditional $*$-automorphism (Theorem \ref{asymptotic}). 
 Building on this, Theorem \ref{thm-even-mu-discrete} provides a key result of this work: ``Any discrete semigroup $\{\Phi^n\}_{n \in \mathbb{N}}$ of unital channels is eventually mixed unitary.''


Section 4 extends this analysis to continuous one-parameter semigroups $(\Phi_t)_{t \ge 0}$ of unital quantum channels on $\M{d}$. We establish an analogous result in Theorem \ref{thm-semigp-eve-MU}, proving that these semigroups are also eventually mixed unitary. Furthermore, we demonstrate in Theorem \ref{thm-No-univ-MUindex} that there is no universal upper bound for the mixed unitary index in a fixed dimension $d \ge 3$. That is, no $N \ge 1$ exists such that $\Phi^N$ is a mixed unitary channel for every unital channel $\Phi$ on $\M{d}$.

In Section 5, we examine the local behavior of a continuous semigroup of a unital quantum channels $\{\Phi_t\}_{t\geq 0}$ near origin. First, a characterization theorem is established for such a semigroup on $\M{d}$ to be mixed unitary for all times (Theorem \ref{thm-schonberg-mixed unitary}). Using this Theorem, we prove that every continuous one parameter semigroup $\{\Phi_t\}_{t\geq 0}$ of unital quantum channels on $\M{d}$ is either mixed unitary for all times or there is a time $t_0$ such that $\Phi_t$ is never mixed unitary for all times $t\in(0,t_0)$ (Theorem \ref{thm-semigroup-not mu-intially}).  We end the Section with several examples of unital quantum channels with some explicit computations.

In the last section, we  first discuss mixed Weyl-unitary channels on $\M{d}$. We provide a complete structure theorem (Theorem \ref{thm-mixed-Weyl}) for a one-parameter semigroup $\{\Phi_t\}_{t\geq 0}$ of unital quantum channels on $\M{d}$ to be Weyl covariant. We also introduce the notion of $G$-mixed unitary channel on $\M{d}$ for any subgroup $G$ of $\mbb{U}(d)$ and prove a characterization theorem (See Theorem \ref{thm-semi-G-MU}) for a continuous one parameter semigroup  of unital quantum channels on $\M{d}$ to be $G$-mixed unitary for all times and close with our results on Schur maps.


\section{Background}

\subsection{Notations}
Let $\mbb{C}^d$ denote the $d$-dimensional complex Hilbert space with the inner product  linear in the second variable and conjugate linear in the first. 
Let $\M{d_1\times d_2}$ be the algebra of complex matrices of size $d_1\times d_2$, we use $\M{d}$ to refer $\M{d\times d}$. The collection of all positive (semidefinite) matrices  and unitary matrices in $\M{d}$ are denoted by $\M{d}^+$ and $\mbb{U}(d)$ respectively.
For every $x\in\mbb{C}^{d_1}$ and $y\in\mbb{C}^{d_2}$, define the rank one linear map $\ranko{x}{y}:\mbb{C}^{d_2}\to\mbb{C}^{d_1}$ by  $\ranko{x}{y}(z):=\ip{y,z}x$, for all $z\in\mbb{C}^{d_2}$
Additionally, we let $\T=\T_d: \M{d}\to \M{d}$ be the transpose map (with respect to the standard orthonormal  basis
$\{e_i: 1\leq i\leq d\}$).
Let $\ip{X, Y}:=\tr(X^*Y)$ be the Hilbert-Schmidt inner product on $\M{d}$, and $\norm{.}_2$ be the corresponding Hilbert Schmidt norm on $\M{d}$. Then the \it{dual} of a linear map $\Phi:\M{d_1}\to\M{d_2}$ is the unique linear map $\Phi^*:\M{d_2}\to\M{d_1}$  such that $\tr({\Phi(X)}^*Y)=\tr(X^*\Phi^*(Y))$ for all $X\in\M{d_1}$ and $Y\in\M{d_2}$.  Denote the space of all linear maps from $\M{d_1}$ to $\M{d_2}$ by $\B{\M{d_1},\M{d_2}}$.

Trace-preserving completely positive maps are known as quantum channels. Any unital CP map $\Phi$ on $\M{d}$, satisfies the Kadison-Schwarz inequality: $$\Phi(X^*)\Phi(X)\leq \Phi(X^*X), ~\forall X\in\M{d}.$$ 
From this, we see that every unital quantum channel  on $\M{d}$ is contractive with respect to the Hilbert-Schmidt norm. We make repeated use of this fact.
In quantum information theory, the Choi-Jamio\l kowski isomorphism \cite{Cho75, Jam72}, commonly known as channel-state duality, illustrates the connection between quantum channels and quantum states.

\begin{thm}[Choi-Jamio\l kowski \cite{Cho75, Jam72}]\label{thm-C-J-map}
	Let $\msc{J}$ be the linear map from $\B{\M{d_1},\M{d_2}}\to\M{d_1}\otimes\M{d_2}$ defined by 
	$$\Phi\mapsto (\id_{d_1}\otimes \Phi)(\ranko{\Omega_{d_1}}{\Omega_{d_1}}),$$
	where $\Omega_{d_1}=\frac{1}{\sqrt{d_1}}\sum_{j=1}^{d_1}e_j\otimes e_j$.  Then we have the following:
	
	\begin{enumerate}[label=(\roman*)]
		\item $\msc{J}$ is a vector space isomorphism. 
		\item $\Phi$ is Hermitian preserving (i.e., $\Phi(X^*)=\Phi(X)^*,~\forall X\in\M{d}$) if and only if $\msc{J}(\Phi)$ is a Hermitian matrix.
		\item $\Phi$ is positive map if and only if $\msc{J}(\Phi)$ is positive on simple tensor vectors, that is, $\ip{x\otimes y, \msc{J}(\Phi)x\otimes y}\geq 0$ for all $x\in\mbb{C}^{d_1}$ and $y\in\mbb{C}^{d_2}$. 
		\item $\Phi$ is CP if and only if $\msc{J}(\Phi)$ is positive. 
             \item $\Phi$ is unital if and only if $(\tr\otimes\id_{d_2})(\msc{J}(\Phi))=\frac{I_{d_2}}{d_1}$.
             \item $\Phi$ is trace-preserving if and only if $(\id_{d_1}\otimes\tr)(\msc{J}(\Phi))=\frac{I_{d_1}}{d_1}$.
	\end{enumerate}
\end{thm}

The matrix $\msc{J}(\Phi )$ is  known as the Choi matrix of $\Phi .$ 

\begin{cor}[Choi-Kraus \cite{Cho75,Kra71}]\label{cor-choi-kraus-rep}
	Let $\Phi:\M{d_1}\to\M{d_2}$ be a linear map. Then $\Phi$ is CP if and only if there exists $\{V_j\}_{j=1}^n\subseteq\M{d_1\times d_2}$ such that  
	\begin{align}\label{eq-Kraus-decompo}
		\Phi=\sum_{j=1}^n\mathrm{Ad}_{V_j},
	\end{align}
	where $\mathrm{Ad}_V(X):=V^*XV$ for all $X\in\M{d_1}$.
\end{cor}

The representation given by \eqref{eq-Kraus-decompo} is known as the \emph{Kraus decomposition} of the map $\Phi$, and the operators $V_i$'s are called the Kraus operators. The minimal number of Kraus operators required to represent $\Phi$ is called the \emph{Choi-rank} of $\Phi$, and it is equal to the rank of the Choi-matrix $\msc{J}(\Phi)$. Note that the Kraus decomposition of $\Phi$ need not be unique. However, if $\Phi$ is expressed in two different Kraus decomposition $\sum_{j=1}^n\mathrm{Ad}_{V_j}$ and  $\sum_{i=1}^m\mathrm{Ad}_{W_i}$, then $\lspan\{V_j:1\leq j\leq n\}=\lspan\{W_i: 1\leq i\leq m\}$.

\begin{defn}
 An unital quantum channel $\Phi:\M{d} \to \M{d}$  is said to be \it{mixed unitary} if there exist  $\{\lambda_j\}_{j=1}^{\ell}\subseteq [0,1]$ 
 and unitaries $\{U_j\}_{j=1}^\ell\subseteq\mbb{U}(d)$ such that $\sum_{j=1}^\ell\lambda_j=1$ and
 \begin{align*}
     \Phi(X)=\sum_{i=1}^{\ell}\lambda_i U_i^{*}X U_i, ~~~~~~~\forall X\in \M{d}.
 \end{align*}
\end{defn}
The \it{mixed unitary rank} of $\Phi$ is the minimal $\ell$ for which such decomposition exists. It is immediate that, Choi-rank of $\Phi$ is less than or equal to the mixed unitary rank of $\Phi$. In \cite{Bus06}, it is proved that the mixed unitary rank of $\Phi$ is at most the square of the Choi-rank of $\Phi$, which is less than or equal to $d^4$. For more details on mixed unitary rank, see \cite{CKLi22}.

 It is well-known that every unital quantum channel on $\M{2}$ is  mixed unitary and this is not true in general for maps on $\M{d}$ with $d\geq 3$. For ready reference we have the following example.
 
 \begin{eg}\cite[Example 4.3]{Wat18} (Holevo-Werner channel) \label{eg-not-mu}
    Let $\Phi:\M{3} \to \M{3}$ be the  unital quantum channel given by 
 $  \Phi(X):=\frac{1}{2}{(tr(X)I_3-X^T)}, \quad \forall X\in\M{3}.$ 
 Then $\Phi$ is not a mixed unitary.
 \end{eg}

Let $\msc{H}(d)$ and $\msc{MU}(d)$ be the set of all Hermitian preserving linear maps on $\M{d}$ and the collection of all mixed unitary channels on $\M{d}$, respectively. The Choi-Jamio\l kowski isomorphism $\msc{J}$ (c.f. \ref{thm-C-J-map}) establishes a bijective correspondence between $\B{\M{d},\M{d}}$  and $\M{d}\otimes\M{d}$. Define an inner product on  $\B{\M{d}}:=\B{\M{d},\M{d}}$ by \begin{equation}\label{inner product}
\ip{\Phi,\Psi} := \tr\big(\msc{J}(\Phi)^*\msc{J}(\Psi)\big). \end{equation} Then, $\msc{J}$ becomes an isometric isomorphism between the Hilbert spaces $\B{\M{d}}$ and $\M{d}\otimes\M{d}$.  Furthermore, we have the following:

       \begin{enumerate}[label=(\roman*)]
          \item The set $\msc{H}(d)$ is isometrically isomorphic to the subset of $\M{d}\otimes\M{d}$ consisting of Hermitian matrices, denoted by $(\M{d}\otimes\M{d})_{Her}$;

          \item The image of $\msc{MU}(d)$ under the Choi-Jamio\l kowski isomorphism, denoted by $\msc{J}(\msc{MU}(d))$, is equal to the convex hull of Choi matrices corresponding to unitary channels. i.e., $\msc{J}(\msc{MU}(d))=\text{co}\{\msc{J}({\mathrm{Ad}_U}): U\in\mbb{U}(d)\}$ and this subset is contained within the set of all positive matrices $\rho\in \M{d}\otimes\M{d}$ such that  $(\id\otimes\tr)(\rho) = (\tr\otimes\id)(\rho) = \frac{1}{d}I_d$.
       \end{enumerate}

Recall that a subset $\mcl{C}$ of a real topological vector space $\mcl{V}$ is a convex cone if it is a convex set and closed under non-negative scalar multiplication. Let $\mcl{V}^{'}$ denote the set of all continuous functionals on $\mcl{V}$. 

\begin{defn}
    Let $\mcl{C}$ be any subset of $\mcl{V}$. The \emph{dual} of $\mcl{C}$, denoted and defined as
    \begin{align}
        \mcl{C}^\circ:=\{f\in\mcl{V}^{'}:f(c)\geq 0, \forall c\in \mcl{C}\}
    \end{align}
\end{defn}
Note that $\mcl{C}^\circ$ is a closed convex cone. Furthermore, if $\mcl{V}$ is the locally convex Hausdorff space, then the double dual of a set $\mcl{C}$ is equal to the closed convex cone generated by $\mcl{C}$.  The following lemma is known in the literature, but we include the details for completeness. 
\begin{lem}\label{lem-interior-dual}   Consider $\msc{MU}(d)$ as a subset of the real vector space $\msc{H}(d)$. Then the following holds:
    \begin{enumerate}[label=(\roman*)]
        \item  $\msc{MU}(d)$ is a compact convex set and is closed under compositions.
        \item $\msc{MU}(d)^{\circ}=\{\Gamma\in\msc{H}(d): \ip{\Gamma,\mathrm{Ad}_U}\geq 0~~\forall  U\in\mbb{U}(d)\}$.
        \item  Suppose $\Phi\in\msc{H}(d)$ is unital and trace-preserving. Then $\Phi\in\msc{MU}(d)$ if and only if $\ip{\Gamma,\Phi}\geq 0$ for all unital and trace-preserving  $\Gamma\in\msc{MU}(d)^{\circ}$.
         \item The interior of the set $\msc{MU}(d)$ is empty.
    \end{enumerate}
    
\end{lem}

\begin{proof}
       $(i)$  Since the set of unitary channels is compact and $\msc{MU}(d)$ is its convex hull it is also compact. Clearly it is closed under compositions.\\

$(ii)$  Note that, any linear functional $f$ on $\msc{H}(d)$ is of the form $f(\Phi)=\tr\big(\msc{J}(\Gamma)\msc{J}(\Phi)\big)=\ip{\Gamma,\Phi}$, for some $\Gamma\in\msc{H}(d)$. Therefore, from the definition of the dual cone,
      \begin{align*}
          \msc{MU}(d)^{\circ}&=\{f\in(\msc{H}(d))^{'}: f(\Phi)\geq 0, ~~\forall \Phi\in\msc{MU}(d)\}\\
                             &=\{\Gamma\in\msc{H}(d):\ip{\Gamma,\Phi}\geq 0, ~~\forall \Phi\in\msc{MU}(d)\}\\
                             &=\{\Gamma\in\msc{H}(d):\ip{\Gamma,\mathrm{Ad}_U}\geq 0, ~~\forall U\in\mbb{U}(d)\}.
      \end{align*}
\\
$(iii)$  Assume that $\Phi$ is mixed unitary, that is $\Phi=\sum_{j=1}^\ell\lambda_j\mathrm{Ad}_{U_j}$ for some $U_j\in\mbb{U}(d)$ and $\lambda_j\in[0,1]$, then for any unital, trace-preserving $\Gamma\in\msc{MU}(d)^\circ$, we have 
\begin{align*}
    \ip{\Gamma,\Phi}&=\ip{\Gamma,\sum_{j=1}^\ell\lambda_j\mathrm{Ad}_{U_j}}=\sum_{j=1}^\ell\lambda_j\ip{\Gamma,\mathrm{Ad}_{U_j}}\geq 0.
\end{align*}
Conversely, assume that $\ip{\Gamma,\Phi}\geq 0$ for all unital and trace-preserving  $\Gamma\in\msc{MU}(d)^\circ$. Let $\delta_d$ be the depolarizing channel as in  \eqref{depolarizing}, and consider the sets,
\begin{align*}
    \mcl{M}:&=\{\Psi- 
       \delta_d:\Psi\in\msc{MU}(d)\},\\
   \mcl{N}:&=\{\Psi- 
       \delta_d:\Psi\in\msc{H}(d)  \text{ is  unital and trace-preserving} \}.
\end{align*}
Since $\msc{MU}(d)$ is a closed convex set, we have $\mcl{M}$ is a closed convex subset of a real vector space $\mcl{N}$. 
Suppose $\Phi\notin\msc{MU}(d)$, then  $\msc{J}(\Phi-\delta_d)\in \msc{J}(\mcl{N})\setminus \msc{J}(\mcl{M})$. Note that  any linear functional on $\msc{J}(\mcl{N})$ is of the form $Z\mapsto \tr(YZ)$ for some $Y\in\msc{J}(\mcl{N})$. Therefore, by the Hahn-Banach separation theorem  
there exist $Y\in\msc{J}(\mcl{N})$ such that 
\begin{align*}
    \tr(YZ)\leq 1 < \tr(Y\msc{J}(\Phi-\delta_d)),~~~~~\forall Z\in \msc{J}(\mcl{M}).
\end{align*}
 Take $\Gamma:=\msc{J}^{-1}(W)$, where $W=\frac{1}{d^2}(I_{d^2}-Y)$. Then $\Gamma\in\msc{H}(d)$. Now as $(\id\otimes\tr)(W)=\frac{I_d}{d}=(\tr\otimes\id)(W)$, we have $\Gamma$ is  unital and trace-preserving. Furthermore, for every unitary $U\in\mbb{U}(d)$,
      \begin{align*}
            \ip{\Gamma,\mathrm{Ad}_U}=\tr(\msc{J}(\Gamma)\msc{J}(\mathrm{Ad}_U))&=\tr(W\msc{J}(\mathrm{Ad}_U))\\
                                                                      &=\frac{1}{d^2}\{\tr(\msc{J}(\mathrm{Ad}_U))-\tr(Y\msc{J}(\mathrm{Ad}_U))\}\\
                                                                      &=\frac{1}{d^2}\{1-\tr(Y\msc{J}(\mathrm{Ad}_U))+\frac{1}{d^2}\tr(Y)\} \qquad      (\because \tr(Y)=0,& \tr(\msc{J}(\mathrm{Ad}_U))=1)\\
                                                                     &=\frac{1}{d^2}\{1-\tr\left(Y(\msc{J}(\mathrm{Ad}_U)-\frac{1}{d^2}I_{d^2})\right)\}\\
                                                                     &=\frac{1}{d^2}\{1-\tr\left(Y\big(\msc{J}(\mathrm{Ad}_U-\delta_d)\big)\right)\}\\
                                                                     &\geq0
        \end{align*}
        This shows that $\Gamma\in\msc{MU}(d)^\circ$. But,
      \begin{align*}
          \ip{\Gamma,\Phi}=\tr(\msc{J}(\Gamma)\msc{J}(\Phi))=\tr(W\msc{J}(\Phi))&=\frac{1}{d^2}\{\tr(\msc{J}(\Phi))-\tr(Y\msc{J}(\Phi))\}\\
                                                 &=\frac{1}{d^2}\{1-\tr(Y\msc{J}(\Phi))+\frac{1}{d^2}\tr(Y)\}\\
                                                 &=\frac{1}{d^2}(1-\tr(Y\msc{J}(\Phi-\delta_d)))<0. 
      \end{align*}
          This contradicts our assumption.\\    
$(iv)$ Suppose the interior of $\msc{MU}(d)$ is non-empty. i.e.,  there exists a $\Phi_0\in\msc{MU}(d)$ and $r>0$ such that $B(\Phi_0,r):=\{\Psi\in\msc{H}(d):\norm{\Psi- 
        \Phi_0}_2<r\}$\footnote{Since any two norms are equivalent in finite dimensional space, so we can choose the ball in Hilbert-Schmidt norm.} is a subset of $\msc{MU}(d)$. Here $\norm{\Psi}_2:=\tr\big(\msc{J}(\Psi)^*\msc{J}(\Psi)\big)$. Observe that  $\msc{J}(B(\Phi_0,r)-\delta_d)\subseteq \mcl{X}$, where $\mcl{X}=\{\rho\in (\M{d}\otimes\M{d})_{Her}: \id\otimes\tr(\rho)=\tr\otimes \id (\rho)=0\}$  and $\delta_d$ be the depolarizing channel as in \eqref{depolarizing}. Since $\msc{J}$ is an isometry and the translation of an open set is open, we deduce that $\mcl{X}$ contains an open set. Therefore, $\mcl{X}=(\M{d}\otimes\M{d})_{Her}$, which is not possible.  
\end{proof}
 
\section{Discrete semigroups  of mixed unitary channels}

Example \ref{eg-not-mu} shows that for $d \geq 3$, not every unital quantum channel $\Phi$ on $\M{d}$ is a mixed unitary. However,  Theorem 21 of \cite{KJRM23} shows that if
 $\Phi:\M{d}\to\M{d}$ is a unital quantum channel,  then there exists some $k\in \mathbb{N}$ such that $\Phi^k$ is a mixed unitary channel. This suggests the following definition. 
\begin{defn}
Let $\{\Phi _k\}_{k\in \mathbb{N}}$ be a sequence of CP maps on $\M{d}.$ (i) It is said to be {\em frequently mixed unitary\/} if it has a subsequence of mixed unitary channels. (ii) It is said to be {\em eventually
 mixed unitary\/} if there exists $k\in \mathbb{N}$ such that $\Phi _n$ is mixed unitary for $n\geq k.$
\end{defn}

Clearly, if $\Phi ^k$ is a mixed unitary so is $\Phi ^{mk}$ for all $m\geq 1.$ Hence the main result of \cite{KJRM23} is telling us that for any unital quantum channel $\Phi$, $\{\Phi ^n\}_{n\in \mathbb{N}}$ is frequently mixed unitary. In this article (See Theorem \ref{thm-even-mu-discrete}), we make the stronger assertion that it is eventually mixed unitary. This was known in some special cases, such as \it{primitive} (meaning that it is \it{irreducible} and has a trivial \it{peripheral spectrum}) channels (\cite{HRS20, KJRM23}) and here we have it in general. The main tool we need is a result from \cite{KJRM23} which says that quantum channels fixing elements of a subalgebra $\mcl{A}$ and sufficiently close to the conditional expectation map on to $\mcl{A}$  are mixed unitary.  We begin with relevant definitions and results.

\begin{defn}
    Let $\mcl{A}$ be a unital $C^*$-subalgebra of $\M{d}$. A CP map $\mbf{E} :\M{d}\to  \M{d}$ with $\mbf{E}(\M {d})=\mcl{A}$ is said to be a conditional expectation onto $\mcl{A}$, if $\mbf{E} (A)=A, \forall A\in \mathcal{A}$.
\end{defn}    
Note that a conditional expectation map $\mbf{E}$ as above is of norm 1 as it is a unital CP map and it is also a projection ($\mbf{E}^2=\mbf{E}$) as it is the identity map on its range. It is also known that $\mbf{E}(AXB)=A\mbf{E}(X)B$ for all $A,B\in \mcl{A}$ and $X\in \M{d}.$ See \cite{Ume53, Ra06}.

 While a conditional expectation onto a $C^*$-subalgebra $\mcl{A}$ of $\M{d}$ is not necessarily unique, there exists a unique trace-preserving conditional expectation. 
Denote it by $\mbf{E}_{\mcl{A}}$. It is equal to the orthogonal projection of $\M{d}$ onto $\mcl{A}$ with respect to  the Hilbert-Schmidt inner product on $\M{d}$. In \cite{BhBha25} one can see a concrete way of writing $\mbf{E}_{\mcl{A}}$ as a mixed unitary channel. 
 One of the fundamental quantum channels widely discussed in the literature is the \it{completely depolarizing channel} $\delta_d$ as defined in
 (\ref{depolarizing}).
%
Note that $\delta _d$ is the unique trace-preserving conditional expectation onto the algebra $\mbb{C}I_d$. The following Proposition is due to John Watrous \cite{Wat09}. It says that all unital quantum channels sufficiently close to the depolarizing channel are mixed unitary.

\begin{prop} \cite[Corollary 2]{Wat09} \label{lem-dual-MU}  Fix $d\in \mathbb{N}.$
 There exists  $r>0$ depending on $d$  such that, every unital, quantum channel $\Phi$ on $\M{d}$ with $\norm{\Phi-\delta_d}<r$, is mixed unitary. 
 \end{prop}

 In 2024, David Kribs et al. \cite{KJRM23} found a clever way of extending this result with the depolarising channel replaced by conditional expectation maps on to arbitrary unital $C^*$-subalgebras of $\M{d}$. For a given a unital $C^*$-subalgebra $\mcl{A}$ of $\M{d}$, we let $\mcl{F}(\mcl{A})$ to be the set of all unital quantum channels which leave elements of $\mcl{A}$ fixed. That is,
 \begin{equation*}
     \mcl{F}(\mcl{A}):=\{\Phi:\M{d}\to\M{d}: \Phi \mbox{ is unital quantum channel such that } \Phi(A)=A,~~\forall A\in \mcl{A}\}.
 \end{equation*}
 It is easy to see that $\mcl{F}(\mcl{A})$ is closed under the compositions. Furthermore, it is closed under adjoints also. This follows from the fact that the fixed point space of $\Phi$ is equal to the commutant of the set $\{V_j:1\leq j\leq \ell\}$, where $\{V_j\}_{j=1}^\ell$ is any set of Kraus operators of $\Phi$ (See \cite[Theorem 2.1]{Kri03}).

The following result has been observed and used in \cite{KJRM23}. However, it is not explicitly stated or proved.  The result is essential for us and so we provide the details. 

\begin{thm}\cite[Theorem 4, Lemma 12]{KJRM23}\label{thm-cond-exp-mu}
    Let $\mcl{A}$ be a unital $C^*$-subalgebra of $\M{d}$ and let $\mbf{E}_{\mcl{A}}$ be the unique trace-preserving conditional expectation  onto $\mcl{A}$. Then there exists $r_{\mcl{A}}>0$ such that, if $\Phi\in\mcl{F}(\mcl{A})$ with $\norm{\Phi-\mbf{E}_{\mcl{A}}}\leq r_{\mcl{A}}$, then $\Phi$ is mixed unitary. 
\end{thm} \begin{proof}
Lemma 12 of \cite{KJRM23} proves that $\mbf{E}_{\mcl{A}}$ is mixed unitary. Theorem 4 of the same article shows that there exists a $p\in(0,1)$ such that $p\Phi+(1-p)\mbf{E}_{\mcl{A}}$ is mixed unitary, for all $\Phi\in\mcl{F}(\mcl{A})$. Now, for a given $\Phi\in\mcl{F}(\mcl{A})$, define $\Phi_{p}:=\mbf{E}_{\mcl{A}}+\frac{1}{p}(\Phi-\mbf{E}_{\mcl{A}})$. Clearly, $\Phi_p$ fixes elements of $\mcl{A}$. Suppose $\Phi_p$ is CP, then $\Phi_p\in\mcl{F}(\mcl{A})$, and hence $p\Phi_{p}+(1-p)\mbf{E}_{\mcl{A}}=\Phi$ is mixed unitary. Therefore, it suffices to prove that there exists $r_{\mcl{A}}> 0$ such that  $\norm{\Phi-\mbf{E}_{\mcl{A}}}\leq r_{\mcl{A}}$, then $\Phi_p$ is CP. For this, without loss of generality assume that $\mcl{A}=\bigoplus_{k=1}^\ell (I_{m_k}\otimes \M{n_k})$, where 
   $$I_{m_k}\otimes \M{n_k}:=\{I_{m_k}\otimes X\in\M{m_kn_k}: X\in\M{n_k}\}$$ 
and $n_k,m_k\in\mbb{N}$ such that $\sum_{k=1}^\ell m_kn_k=d$.    For computational simplicity, we will only prove the case when $\ell=2$, the general case follows similarly.  Let $\mcl{H}=\mcl{H}_1\oplus \mcl{H}_2,$ where $H_k=\mbb{C}^{m_k}\otimes \mbb{C}^{n_k}$, for $k=1,2$ and  $\bigcup_{k=1}^2\{e_r^{(k)}\otimes f_s^{(k)}: 1\leq r\leq m_k, 1\leq s\leq n_k\}$ be the standard orthonormal basis of $\mcl{H}$. Then, the associated  matrix units of $\B{\mcl{H}}$ are $\bigcup_{j,k=1}^2\{E_{rr'}^{(j,k)}\otimes F_{ss'}^{(j,k)}: 1\leq r\leq m_j,1\leq r'\leq m_k,1\leq s\leq n_j, 1\leq s'\leq n_k\}$, where $E_{rr'}^{(j,k)}=\ranko{e_r^{(j)}}{e_{r'}^{(k)}}$ and $F_{ss'}^{(j,k)}=\ranko{f_s^{(j)}}{f_{s'}^{(k)}}$. Write any $A\in\B{\mcl{H}}$ in the block form as
$A=\Matrix{A_{11}&A_{12}\\A_{21}&A_{22}}$, with $A_{jk}\in\B{\mcl{H}_k,\mcl{H}_j}$, for $j,k=1,2$. Note that, the map  $\mbf{E}_{\mcl{A}}$ is given by 
   \begin{align*}
       \mbf{E}_{\mcl{A}}\big(\Matrix{A_{11}&A_{12}\\A_{21}&A_{22}})=\Matrix{\frac{I_{m_1}}{m_1}\otimes (\tr\otimes \id)(A_{11})&0\\0&\frac{I_{m_2}}{m_2}\otimes (\tr\otimes \id)(A_{22})}.
   \end{align*}
   Let $\{K_j\}_{j=1}^N$ be a set of Kraus operators of $\Phi$. Since $\Phi\in\mcl{F}(\mcl{A})$, we get  that  
  \begin{align*}
      K_j= \Matrix{K_{j1}\otimes I_{n_1}&0\\0& K_{j2}\otimes I_{n_2}}, ~~\forall 1\leq j\leq N,
  \end{align*} 
  with $K_{ji}\in\M{m_i}$ such that $\sum_{j=1}^N K_{ji}^*K_{ji}=\sum_{j=1}^N K_{ji}K_{ji}^*=I_{m_i}$, for $i=1,2$. Therefore, the map $\Phi$ is given by
\begin{align*}
    \Phi\big(\Matrix{A_{11}&A_{12}\\A_{21}&A_{22}}\big)=\Matrix{\sum_{j=1}^N (K_{j1}^*\otimes I_{n_1})A_{11}(K_{j1}\otimes I_{n_1})&\sum_{j=1}^N(K_{j1}^*\otimes I_{n_1})A_{12}(K_{j2}\otimes I_{n_2})\\ \sum_{j=1}^N (K_{j2}^*\otimes I_{n_2}A_{21}(K_{j1}\otimes I_{n_1})&\sum_{j=1}^N (K_{j2}\otimes I_{n_2})A_{22}(K_{j2}\otimes I_{n_2})}.
\end{align*}
The Choi-matrix of $\Phi$ is seen to be

    \resizebox{.98\linewidth}{!}{
		\begin{minipage}{\linewidth}
\begin{align*}
  \msc{J}(\Phi)&= \frac{1}{d}\Big(\sum_{j,k=1}^2\sum_{r,r'=1}^{m_j,m_k}\sum_{s,s'=1}^{n_j,n_k}E_{rr'}^{(j,k)}\otimes F_{ss'}^{(j,k)}\otimes \Phi\big(E_{rr'}^{(j,k)}\otimes F_{ss'}^{(j,k)}\big)\Big)\\ 
&=\frac{1}{d}\Matrix{\sum_{r,r'=1}^{m_1}\sum_{s,s'=1}^{n_1}e_{rr'}^{(1,1)}\otimes f_{ss'}^{(1,1)}\otimes \sum_{j=1}^N (K_{j1}^*e_{rr'}^{(1,1)}K_{j1})\otimes f_{ss'}^{(1,1)}& \sum_{r,r'=1}^{m_1,m_2}\sum_{s,s'=1}^{n_1,n_2}e_{rr'}^{(1,2)}\otimes f_{ss'}^{(1,2)}\otimes \sum_{j=1}^N (K_{j1}^*e_{rr'}^{(1,2)}K_{j2})\otimes f_{ss'}^{(1,2)}\\\sum_{r,r'=1}^{m_2,m_1}\sum_{s,s'=1}^{n_2,n_1}e_{rr'}^{(2,1)}\otimes f_{ss'}^{(2,1)}\otimes \sum_{j=1}^N (K_{j2}^*e_{rr'}^{(2,1)}K_{j1})\otimes f_{ss'}^{(2,1)}&\sum_{r,r'=1}^{m_2}\sum_{s,s'=1}^{n_2}e_{rr'}^{(2,2)}\otimes f_{ss'}^{(2,2)}\otimes \sum_{j=1}^N (K_{j2}^*e_{rr'}^{(2,2)}K_{j2})\otimes f_{ss'}^{(2,2)}},   
\end{align*}
\end{minipage}}
 where $e_{rr'}^{(j,k)},f_{ss'}^{(j,k)}$ are the standard matrix units of $\M{m_j\times m_k}$ and $\M{n_j\times n_k}$ respectively. Note that if $\mcl{K}=\mcl{K}_1\oplus \mcl{K}_2$, then the matrix units $e_{rr'}$ in $\B{\mcl{K}_1}$ is treated as $E_{rr'}=\Matrix{e_{rr'}&0\\0&0}$ in $\B{\mcl{K}}$.
Now, by swapping the second and third tensor products in each block entry, the Choi-matrix is equal to 

\begin{align*}
    \frac{1}{d}\Matrix{A\otimes \ranko{f^{(1)}}{f^{(1)}}& B\otimes \ranko{f^{(1)}}{f^{(2)}}\\C\otimes \ranko{f^{(2)}}{f^{(1)}}&D\otimes\ranko{f^{(2)}}{f^{(2)}}},    
\end{align*}
where
\begin{align*}
A:&=\sum_{r,r'=1}^{m_1}e_{rr'}^{(1,1)}\otimes\sum_{j=1}^N (K_{j1}^*e_{rr'}^{(1,1)}K_{j1}),\qquad
    B:=\sum_{r,r'=1}^{m_1,m_2}e_{rr'}^{(1,2)}\otimes\sum_{j=1}^N (K_{j1}^*e_{rr'}^{(1,2)}K_{j2})\\
    C:&=\sum_{r,r'=1}^{m_2,m_1}e_{rr'}^{(2,1)}\otimes\sum_{j=1}^N (K_{j2}^*e_{rr'}^{(2,1)}K_{j1}),\qquad
    D:=\sum_{r,r'=1}^{m_2}e_{rr'}^{(2,2)}\otimes\sum_{j=1}^N (K_{j2}^*e_{rr'}^{(2,2)}K_{j2}),
\end{align*}
and $f^{(j)}=\sum_{s}f_s^{(j)}\otimes f_s^{(j)}$, for $j=1,2$. Note that $\frac{1}{m_1+m_2}\Matrix{A&B\\C&D}$ is the Choi-matrix (upto permutation equivalence) of the unital quantum channel $\Psi:\M{m_1+m_2}\to\M{m_1+m_2}$ given by 
\begin{align*}
    \Psi(\Matrix{X_{11}&X_{12}\\X_{21}&X_{22}})=\sum_{j=1}^N\Matrix{K_{j1}^*&0\\0&K_{j2}^*}\Matrix{X_{11}&X_{12}\\X_{21}&X_{22}}\Matrix{K_{j1}&0\\0&K_{j2}}.
\end{align*}
Thus, by letting $\Psi':=(m_1+m_2)\Psi$, the Choi-matrix (up to permutation equivalence) of $\Phi$ is equal to the tensorial Schur product (See \cite{VsSu15}) of $ \frac{1}{d}\msc{J}(\Psi')$ and $\ranko{f}{f}$. That is,
\begin{align*}
    \frac{1}{d}\msc{J}(\Psi')\circ^\otimes \ranko{f}{f},
\end{align*}
where $f=f^{(1)}\oplus f^{(2)}$. It is a principal submatrix of $\frac{1}{d}\msc{J}(\Psi' )\otimes \ranko{f}{f}$.
Using the similar argument, we also have that the Choi-matrix (upto permutation equivalence) of $\mbf{E}_{\mcl{A}}$ is equal to
\begin{align*}
    \frac{1}{d}\Matrix{\frac{I_{m_1}\otimes I_{m_1}}{m_1} &0\\0&\frac{I_{m_2}\otimes I_{m_2}}{m_2}}\circ^\otimes \ranko{f}{f}.
\end{align*}
Therefore, the Choi-matrix (upto permutation equivalence) of $\Phi_p$ is $\frac{1}{d}\big(Z+\frac{1}{p}(\msc{J}(\Psi')-Z)\big)\circ^\otimes \ranko{f}{f},$ where $Z=\Matrix{\frac{I_{m_1}\otimes I_{m_1}}{m_1} &0\\0&\frac{I_{m_2}\otimes I_{m_2}}{m_2}}$.
 Now, as $Z$ is strictly positive, taking  $r$ as the smallest positive eigenvalue of $Z$,  $Z+\Delta$ is positive,  
 whenever $\Delta=\Delta^*$ with $\norm{\Delta}\leq r$. Since  a  principal submatrix has norm less than or equal to the original matrix, we get 
$\norm{\msc{J}(\Phi)-\msc{J}(\mbf{E}_{\mcl{A}})}\geq \frac{1}{d}\norm{\msc{J}(\Psi')-Z}$. Set $r_\mcl{A}= \frac{pr}{d}$, then whenever $\Phi\in\mcl{F}(\mcl{A})$ with $\norm{\msc{J}(\Phi)-\msc{J}(\mbf{E}_{\mcl{A}})}\leq r_{\mcl{A}}$  we have $$\frac{1}{p}\norm{\msc{J}(\Psi')-Z}=\frac{d}{ p}\norm{\msc{J}(\Phi)-\msc{J}(\mbf{E}_{\mcl{A}})}\leq r.$$
Hence $\big(Z+\frac{1}{p}(\msc{J}(\Psi')-Z)\big)\geq 0$. By \cite[Proposition 1.3]{VsSu15}, $\frac{1}{d}\big(Z+\frac{1}{p}(\msc{J}(\Psi')-Z)\big)\circ^\otimes \ranko{f}{f}$ is positive. Further, this is equivalent to the map $\Phi_p$  being CP. This completes the proof.
\end{proof}

In general an automorphsim of a subalgebra $\mathcal{A}$ of $M_d$ need not be implemented by a unitary. For instance, take
$\mathcal{A}=\{ a(E_{11}+E_{22})+bE_{33}: a,b\in \mathbb{C}\}\subseteq \M{3}$, where $E_{ij}$ denotes matrix units and consider the automorphism $a(E_{11}+E_{22})+bE_{33}\mapsto b(E_{11}+E_{22})+aE_{33}.$ However, there are no such issues if the automorphism preserves the trace as the following result shows.

\begin{lem}\label{thm-stru-automor-finite}
    Let $\mcl{A}$ be a unital $C^*$-subalgebra of $\M{d}$. Let $\Psi$ be a trace-preserving $*$-automorphism of $\mcl{A}$. Then there exists a unitary $U\in\M{d}$ such that $\Psi(X)=\mathrm{Ad}_{U}(X), ~~\forall X\in \mcl{A}$ and $\mathrm{Ad}_U$ leaves $\mcl{A}^\perp$ invariant, where the orthogonal complement is taken with respect to the Hilbert-Schmidt inner product. 
\end{lem}

\begin{proof}
   Without loss of generality assume that $\mcl{A}=\oplus_{k=1}^\ell (I_{m_k}\otimes \M{n_k})$, where 
   $$I_{m_k}\otimes \M{n_k}:=\{I_{m_k}\otimes X\in\M{m_kn_k}: X\in\M{n_k}\}$$ and $n_k,m_k\in\mbb{N}$ such that $\sum_{k=1}^\ell n_km_k=d$. (Note that $m_k, n_k$ can be repeating.) Each summand
     \begin{align*} J_{i}:=\bigoplus_{k=1}^{\ell}\delta_{ik}(I_{m_k}\otimes\M{n_k}),
     \end{align*}
     is a minimal two-sided ideal of $\mcl{A}$, and these are the only minimal two-sided ideals of $\mcl{A}$. For any $1\leq k\leq \ell$ and $m\in\{m_1,m_2,\cdots,m_\ell\}$, let
      $$B_k(m):=\{i:J_i=\bigoplus_{k=1}^{\ell}\delta_{ik}(I_{m_k}\otimes\M{n_k}) \text{with }n_i=n_k, m_i=m\}.$$
      Since $\Psi$ is a bijective $*$-homomorphism, $\Psi(J_{i})$ is a minimal two-sided ideal for all $i\in B_k(m)$, and hence equal to $J_{j}$ for some $1\leq j\leq \ell$. By comparing the dimension of $\Psi(J_{i})$ and $J_j$, we have $n_j=n_i=n_k$. Now, $\Psi$ is also trace-preserving, implying that $m_i=m_j=m$. This shows that $j\in B_k(m)$. Furthermore, from the bijectivity of $\Psi$ it follows that, there exists a bijective function $\sigma_k:B_{k}(m)\to B_{k}(m)$ such that $\Psi(J_i)=J_{\sigma_k(i)}$. Since any $*$-autmorphism from $\M{n_k}$ to $\M{n_k}$ is of the form $\mathrm{Ad}_{U}$ for some unitary $U\in\M{n_k}$, we conclude that for every $i\in B_k(m)$ there exists a unitary $U_{\sigma_k(i)}\in\M{n_k}$ such that the map $\Psi\arrowvert_{J_{i}}:J_{i}\to J_{\sigma_k(i)}$ is given by 
      \begin{align*}
          \Psi\left(\bigoplus_{r=1}^{\ell}\delta_{ir}(I_{m_r}\otimes X)\right)&=\bigoplus_{r=1}^{\ell}\delta_{\sigma_k(i)r}(I_{m_r}\otimes U_{\sigma_k(i)}^*XU_{\sigma_k(i)}), ~~ \forall X\in\M{n_k}.
      \end{align*}
      Let  $U_k=\bigoplus_{r=1}^\ell (I_{m_r}\otimes W_{r}),$  where $W_r=U_{\sigma_k(r)}$ if $r\in B_k(m)$, otherwise identity matrix in $\M{n_r}$, and $P_k$ in $\M{d}$ be the permutation matrix that reorders the blocks in $B_k(m)$ according to $\sigma_k$. That is,
      \begin{align*}
          P_k^*\left(\sum_{i\in B_k(m)}\bigoplus_{r=1}^{\ell}\delta_{ir}(I_{m_r}\otimes X_i) \right)P_k=\sum_{i\in B_k(m)}\bigoplus_{r=1}^{\ell}\delta_{\sigma_k(i)r}(I_{m_r}\otimes X_i), ~~~\forall X_i\in\M{n_k}.
      \end{align*}
Then, the  map $\Psi\arrowvert_{\sum_{i\in B_k(m)}J_i}:\sum_{i\in B_k(m)}J_i\to\sum_{i\in B_k(m)}J_i$ is of the form 
      \begin{align*}
          \Psi\left(\sum_{i\in B_k(m)}\bigoplus_{r=1}^{\ell}\delta_{ir}(I_{m_r}\otimes X_i)\right)&=\sum_{i\in B_k(m)}\bigoplus_{r=1}^{\ell}\delta_{\sigma_k(i)r}(I_{m_k}\otimes U_{\sigma_k(i)}^*X_iU_{\sigma_k(i)} )\\
          &=P_k^*\left(\sum_{i\in B_k(m)}\bigoplus_{r=1}^{\ell}\delta_{ir}(I_{m_k}\otimes U_{\sigma_k(i)}^*X_iU_{\sigma_k(i)})\right)P_k\\
          &=P_k^*U_k^*\left(\sum_{i\in B_k(m)}\bigoplus_{r=1}^{\ell}\delta_{ir}(I_{m_r}\otimes X_i)\right)U_kP_k,
      \end{align*}
for all $X_i\in\M{n_k}$. 
    Now, let $n_1, n_2,\cdots n_q$ be distinct numbers of the set $\{n_1,n_2,\cdots,n_\ell\}$ and $m_1,m_2\cdots,m_s$ be the distinct numbers of the set $\{m_1,m_2,\cdots m_\ell\}$ . Then $\mcl{A}=\sum_{k=1}^q\sum_{j=1}^s \mcl{A}_k(m_j)$, where $\mcl{A}_k(m_j):=\sum_{i\in B_k(m_j)}J_i$. Furthermore, for every $(k,j)\neq (k',j')$,  $\mcl{A}_k(m_j)\mcl{A}_{k'}(m_{j'})=0=\mcl{A}_{k'}(m_{j'})\mcl{A}_k(m_j)$. Consequently, one can choose a global permutation matrix $P\in\M{d}$ and a block diagonal unitary $V\in\M{d}$ such that $\Psi(Z)=P^*V^*ZVP$, for all $Z\in\mcl{A}$. Taking $U=VP$, completes the proof. 
\end{proof}

Now we recall some basic notions required  to study  asymptotics of CP maps and quantum channels.  Due to its importance in quantum theory of open systems,  there is extensive literature on this topic. We will follow mostly
\cite{BKT25, BhSa23}. However we  refer the reader to  \cite{BCGPY13, AFK25}  for some  recent works and the citations in them for further information.
\begin{defn}
    Let $\Phi:\M{d}\to\M{d}$ be a linear map.
    \begin{enumerate}[label=(\roman*)]
\item The \emph{multiplicative domain} of  $\Phi $, is defined as  
\begin{align*}
    \msc{M}_{\Phi}:=\{X\in\M{d}: \Phi(XY)=\Phi(X)\Phi(Y) \mbox{ and } \Phi(YX)=\Phi(Y)\Phi(X),\quad ~\forall Y\in\M{d}\}.
\end{align*}

It is known that for a unital CP map $\Phi$, we have 
\begin{align*}
 \msc{M}_{\Phi}=\{X\in\M{d}:\Phi(X^*X)=\Phi(X^*)\Phi(X)\mbox{ and } \Phi(XX^*)=\Phi(X)\Phi(X^*)\}.   
\end{align*}       
\item The \emph{peripheral spectrum} of $\Phi$ be the set of all elements in the spectrum
 $\sigma(\Phi)$ whose absolute value is $1$.  i.e., $\{\lambda\in \sigma(\Phi): \abs{\lambda}=1\}$.
 \item The \emph{peripheral space} of $\Phi$ is defined as   
 $$ \mcl{P}(\Phi):=\lspan\{X:\Phi(X)=\lambda X, \abs{\lambda}=1\}.$$
    \end{enumerate}
\end{defn}

 If $\Phi$ is a unital quantum channel,  making use of Kadison-Schwarz inequality it can be seen that  the peripheral space is contained in the multiplicative domain: $\mcl{P}(\Phi) \subseteq \msc{M}_{\Phi}$.

\begin{defn}
    A linear map $\alpha :\M{d}\to\M{d}$ said to be a 
 \emph {conditional $*$-automorphism\/} if $\alpha = \Pi \circ \mbf{E}_{\mcl{A}}$, where $\mbf{E}_{\mcl{A}}$ is a trace-preserving conditional expectation onto some $C^*$-subalgebra $\mcl{A}$ of $\M{d}$ and $\Pi$ is a $*$-automorphism on $\M{d}$, which commutes with $\mbf{E}_{\mcl{A}}$.
\end{defn}

\begin{thm}\label{asymptotic}
Let $\Phi$ be a unital quantum channel on $\M{d}.$  Then $\{\Phi ^n\}_{n\in \mathbb{N}}$ is   asymptotically a conditional $*$-automorphism. i.e., there exists a conditional $*$-automorphism $\alpha$ on $\M{d}$ such that 
    \begin{align*}
        \lim _{n\to \infty} \norm{\Phi^n-\alpha^n}= 0.
    \end{align*}
\end{thm}
\begin{proof}
    Let $\mcl{P}$ be the peripheral space of $\Phi$. That is, $\mcl{P}=\lspan\{X:\Phi(X)=\lambda X, \abs{\lambda}=1\}$.
   Then, the assumption that $\Phi $ is trace-preserving ensures  that  $\mcl{P}$ is a $C^*$-sub algebra of $\M{d}$ (See \cite{BhSa23}).
    Let $\mbf{E}_{\mcl{P}}$ be the trace-preserving conditional expectation onto $\mcl{P}$.  
    Considering $\M{d}$ as a Hilbert space with Hilbert-Schmidt inner product, $\mbf{E}_{\mcl{P}}$ is just the orthogonal projection on to the subspace $\mcl{P}.$
    Moreover  we have the  following (\cite{BhSa23}, Theorem 3.1 and 3.4):
    \begin{enumerate}[label=(\roman*)]
        \item $\M{d}=\mcl{P}\oplus\mcl{N}$, where $\mcl{N}=\{Y: \lim_{n\to\infty}\Phi^n(Y)= 0\}$ is an orthogonal direct sum with respect to the Hilbert-Schmidt inner product;
        \item $\Phi=\alpha+\beta$, where $\alpha:=\Phi\circ\mbf{E}_{\mcl{P}}:\M{d}\to\mcl{P}\subseteq\M{d} \  \text{ and }\ \beta:=\Phi-\alpha =\Phi \circ (\id - \mbf{E}_{\mcl{P}}):\M{d}\to\mcl{N}\subseteq\M{d}$;
        \item  $\Phi\arrowvert_{\mcl{P}}:\mcl{P}\to\mcl{P}$ is a $*$-automorphism and $\lim_{n\to\infty}\beta^n=0$.
    \end{enumerate}
   As $\Phi\arrowvert_{\mcl{P}}$ is trace-preserving $*$-automorphism, from Lemma \ref{thm-stru-automor-finite},  there exists a unitary $U\in\M{d}$ such that $\mathrm{Ad}_U(\mcl{P})=\mcl{P}$ and
  $     \Phi\arrowvert_{\mcl{P}}=\mathrm{Ad}_{U}\arrowvert_{\mcl{P}}$.
This shows that the map, $\alpha=\mathrm{Ad}_{U}\circ \mbf{E}_{\mcl{P}}=\mbf{E}_{\mcl{P}}\circ \mathrm{Ad}_{U}$ is a conditional $*$-automorphism. Furthermore, 
\begin{align*}
    \norm{\Phi^n-\alpha^n}=\norm{\beta^n}\to 0, ~~~\text{as } n\to\infty.
\end{align*}
Note that $\mathrm{Ad}_U$ is an automorphism of $\M{d}$ and it leaves $\mcl{P}$ and $\mcl{N}$ invariant. 
\end{proof}

Finally we are in a position to state and prove our main result.

\begin{thm}\label{thm-even-mu-discrete}
    Let $\Phi:\M{d}\to\M{d}$ be a unital quantum channel. Then $\{\Phi ^n\}_{n\in \mathbb{N}}$ is eventually mixed unitary, that is, there exists a $N\geq 1$ such that $\Phi^n$ is mixed unitary for all $n\geq N$.
\end{thm}
\begin{proof}
We consider the decomposition $\Phi =\alpha +\beta $ and other notation as in
Theorem \ref{asymptotic}. Note that $\alpha^n=\Phi^n\circ\mbf{E}_{\mcl{P}}$ and $\mathrm{Ad}_{(U^*)^{n}}\circ\alpha^n=\mbf{E}_{\mcl{P}}$ for all $n\geq 1$. Thus by 
taking $\tau_n:=\mathrm{Ad}_{(U^*)^n}\circ\Phi^n$ and $X\in\mcl{P}$, we get
 \begin{align*}
     \tau_n(X)&=\mathrm{Ad}_{(U^*)^n}\circ\Phi^n(X)
     =\mathrm{Ad}_{(U^*)^n}\circ\Phi^n\circ\mbf{E}_{\mcl{P}}(X)
     =\mathrm{Ad}_{(U^*)^n}\circ \alpha^n(X)
     =\mbf{E}_{\mcl{P}}(X)
     =X.
 \end{align*}
That is, $\tau_n\in\mcl{F}(\mcl{P})$ for all $n\geq 1$. Furthermore,
\begin{align*}
    \norm{\beta^n}&=\norm{\Phi^n-\alpha^n}=\norm{\mathrm{Ad}_{(U^*)^n}\circ(\Phi^n-\alpha^n)}
            =\norm{\tau_n-\mbf{E}_{\mcl{P}}}.
\end{align*}
 This implies that $\tau_n\to\mbf{E}_{\mcl{P}}$ as $n\to\infty$. Consider $r_{\mcl{P}}>0$ as in Theorem \ref{thm-cond-exp-mu}, then we get a $N\geq 1$ such that $\norm{\tau_n-\mbf{E}_{\mcl{P}}}\leq r_{\mcl{P}}$ for all $n\geq N$. Now, by Theorem \ref{thm-cond-exp-mu}, $\tau_n$ is mixed unitary for all $n\geq N$. Consequently, $\Phi^n=\mathrm{Ad}_{U^n}\circ\tau_n$ is mixed unitary for all $n\geq N$.
\end{proof}
\begin{defn}
Let $\msc{M}$ be a subset of the convex set of unital quantum channels on
$\M{d}$. 
We call $\msc{M}$ \emph{absorbing} if it is a closed convex set and every
unital quantum channel is eventually in $\msc{M}$. i.e., for every unital quantum channel $\tau$ there exists $N\in\mathbb{N}$ such that $\tau^{n}\in\msc{M}$ for all $n\ge N$.
\end{defn}

\begin{cor}
$\msc{MU}(d)$ is the smallest absorbing subset of the set of unital quantum
channels on $\M{d}$. That is, $\msc{MU}(d)$ is absorbing and is contained in
every absorbing set.
\end{cor}

\begin{proof}
From Lemma~\ref{lem-interior-dual}
and Theorem~\ref{thm-even-mu-discrete} it follows that $\msc{MU}(d)$ is absorbing. It remains to show that
$\msc{MU}(d)\subseteq\msc{M}$ for any arbitrary absorbing set $\msc{M}$. 
Let $U\in\mathcal{U}(d)$. Suppose that $U$ is periodic. i.e.,
$U^{q}=I$ for some $q\in\mathbb{N}$. Then $\mathrm{Ad}_U^{\,q}=\mathrm{id}$,
hence $\mathrm{Ad}_U^{\,1+qm}=\mathrm{Ad}_U$ for every $m\ge 1$. Since
$\msc{M}$ is absorbing, there is an $N$ such that $\mathrm{Ad}_U^{\,n}\in\msc{M}$
for all $n\ge N$. Choosing $m$ so that $1+qm\ge N$ gives
$\mathrm{Ad}_U=\mathrm{Ad}_U^{\,1+qm}\in\msc{M}$. Now, suppose $U$ is not periodic. Being unitary, $U$ diagonalizes as
$U=V\operatorname{diag}(e^{2\pi i\theta_1},e^{2\pi i\theta_2},\cdots, e^{2\pi i\theta_d})V^{*}$ for some unitary $V$ and  $\theta_1,\dots,\theta_d\in[0,1)$. For each $j$, choose a sequence of rationals $(r_j^{(k)})$ with $r_j^{(k)}\to \theta_j$ as $k\to \infty$ and set  $W_k=V\operatorname{diag}(e^{2\pi ir_1^{(k)}},e^{2\pi ir_2^{(k)}},\cdots, e^{2\pi ir_d^{(k)}})V^{*}$. Then, each $W_k\in\mathcal{U}(d)$ is periodic and $W_k\to U$ as $k\to\infty$. By periodic case, $\mathrm{Ad}_{W_k}\in\msc{M}$ for every $k$. Since $U\mapsto \mathrm{Ad}_U$ is continuous and $\msc{M}$ is closed, it follows that $\mathrm{Ad}_U=\lim_k \mathrm{Ad}_{W_k}\in\msc{M}$. Thus, we have $\mathrm{Ad}_U\in\msc{M}$ for all $U\in\mathcal{U}(d)$. As $\msc{M}$ is convex and closed, $
\msc{M}\supseteq
\overline{\operatorname{conv}}\,\{\mathrm{Ad}_U : U\in\mathcal{U}(d)\}
=\msc{MU}(d)
$. This completes the proof.
\end{proof}

\section{Quantum dynamical semigroups of unital quantum channels}
Theorem \ref{thm-even-mu-discrete} establishes that every discrete semigroup of unital quantum channels is eventually mixed unitary. Naturally, this raises the question of whether an analogous result holds for continuous semigroups.
To this end, we start by recalling the relevant definitions and results for quantum dynamical semigroups.

\begin{defn}
      A family of linear maps $\{\Phi_t\}_{t\geq 0}$ on $\M{d}$ is said to be a \it{one parameter semigroup} if 
 $ \Phi_0=\id;  ~~\Phi_t\circ\Phi_s=\Phi_{s+t}$, for all $t,s\geq 0$,
and the map $t\mapsto \Phi_t$ is continuous. A one parameter semigroup $\{\Phi _t\}$ is said to be a \it{quantum dynamical semigroup} if $\Phi _t$ is completely positive for every $t\geq 0.$ It is said to be unital  $\Phi_t(I_d)=I_d$ for every $t$ and a channel if $\Phi _t$ is trace-preserving for every $t.$
  \end{defn}
  
If $\{\Phi_t\}_{t\geq 0}$ is a one parameter semigroup of linear maps on $\M{d}$, then the map $$\mathcal{L} (X):= \lim_{t\to 0^+}\frac{\Phi_t(X)-X}{t}, ~~X\in \M{d},$$ is a well-defined  linear map on $\M{d}$, called the \it{infinitesimal generator}  of the semigroup $\{\Phi_t\}_{t\geq 0}$. Furthermore,   $$\Phi_t(X)=e^{t\mathcal{L}}(X) =\sum _{n=0}^\infty \frac{(t\mathcal{L})^n(X)}{n!},~~\forall X\in\M{d}, ~~\forall t\geq 0.$$ 
The generators of quantum dynamical semigroups is characterized by the famous GKLS theorem (See \cite{GKS76,Lin76,KRP,CP17}), namely, we have the following theorem: 
\begin{thm}\label{thm-cp-semigroup}\cite{GKS76}
Consider a one parameter semigroup of linear maps $\{\Phi_t\}_{t\geq 0}$ on $\M{d}$. Then the following are equivalent:
\begin{enumerate}[label=(\roman*)]
\item $\{\Phi_t \}$ is a quantum dynamical semigroup;        
\item The generator $\mathcal{L}$ of the semigroup $\Phi_t$ is of the form 
$$\mathcal{L}(X)=G^*X+XG+\sum_{j=1}^{n}L_j^*XL_j,$$
 where $G, L_j\in\M{d}, 1\leq j\leq n, n\geq 0$.
 \end{enumerate}
 Further, for a quantum dynamical semigroup $\{\Phi _t\},$ with generator $\mathcal{L} $ as above:  (a)  It is unital iff $\mathcal{L}(I)=0$ iff $G=-iH-\frac{1}{2}\sum _{j=1}^n L_j^*L_j$ for some self-adjoint $H$; (b) It is a quantum channel
 iff $\tr\big(\mathcal{L}(X)\big)=0$ for every $X\in \M{d}$ iff $G=-iH-\frac{1}{2}\sum _{j=1}^nL_jL_j^*.$

In this article, we restrict our focus to unital quantum channels. In view of the preceding theorem, it follows that a linear map $\mathcal{L}:\M{d}\to \M{d}$ is the generator of a quantum dynamical semigroup of unital quantum channels iff it has the form:
\begin{equation}\label{GKLS}
\mathcal{L}(X)=i[H,X]-\frac{1}{2}\sum _{j=1}^n(L_j^*L_jX+XL_j^*L_j-2L_j^*XL_j),~~\forall X\in \M{d},
\end{equation}
 for some $H, L_j, 1\leq j\leq n$ in $\M{d}, ~n\geq 0,$ satisfying $H=H^*, ~\sum _{j=1}^nL_jL_j^*=\sum _{j=1}^nL_j^*L_j.$
\end{thm}

The operators $H, L_j$'s appearing in equation (\ref{GKLS}) are called as GKLS coefficients/parameters of $\mcl{L}$. They are not uniquely determined. However, the extent of non-uniqueness can be described (See page 272 of \cite{KRP}). We simply choose and fix one representation of $\mcl{L}$ as in (\ref{GKLS}) and work with it.

The generators of continuous dynamical semigroups of mixed unitary channels are characterized in \cite{KuMa87}, where it is also shown that in general unital semigroups of quantum channels are not necessarily mixed unitary at all times (See Example \ref{eg-semigp-not-mu-kumm}). This raises the question of their asymptotic behavior: for a given unital dynamical semigroup $\{\Phi_t\}_{t \geq 0}$ on $\mathbb{M}_d$, must there exist a threshold $t_0$ after which $\Phi_t$ remains mixed unitary? In Theorem \ref{thm-semigp-eve-MU}, we prove that this is indeed the case.


\begin{thm}\cite{KuMa87}\label{thm-mu-semi-gene}
    Let $\{\Phi_t\}_{t\geq 0}$ be a quantum dynamical semigroup of unital channels. Then the following are equivalent:
    \begin{enumerate}[label=(\roman*)]
        \item $\Phi_t$ is mixed unitary, for all $t\geq 0$;
        \item  The generator  $\mathcal{L}$ of $\{\Phi _t\}$ is given by
        \begin{equation*}
            \mathcal{L}(X)=i[H,X]-\frac{1}{2}\sum_{j=1}^{n}(L_j^2X+XL_j^2-2L_jXL_j)+\sum_{k=1}^{\ell}\lambda_k(U_k^*XU_k-X),
        \end{equation*} 
        where $H,L_j$ are Hermitian, $U_k$ are unitary,  $\lambda_k\geq 0$ 
         and $n,\ell\geq 0$;
        
        \item The infinitesimal generator $\mathcal{L}$ is in the closed cone generated by $\{\mathrm{Ad}_{U}-\id: U\in\mbb{U}(d)\}$.
    \end{enumerate}
\end{thm} 

We observe that for a generator $\mathcal{L}$ to generate mixed unitary channels at all times it is necessary and sufficient that the 
coefficients $L_j$'s appearing in its GKLS expansion can be taken to be either self-adjoint or scalar multiples of unitary.

We analyze the semigroup using the Jordan decomposition of its generator. For a given $\lambda \in \mathbb{C}$, let $J_r(\lambda) \in \mathbb{M}_r$ denote the Jordan block of size $r \in \mathbb{N}$ associated with the eigenvalue $\lambda$. This matrix is defined by having $\lambda$ at each diagonal entry and $1$ at each entry of the super-diagonal.
Note that for $r=1$, $J_1(\lambda )=[\lambda ]\in \M{1}$. 
The following Lemma is perhaps known in the literature, but we are providing a proof for completeness.  
\begin{lem}\label{lem-JCF-gena}
Let $\mathcal{L}:\M{d}\to\M{d}$ be a linear map and let  $\bigoplus_{k=1}^{m}J_{d_k}(\lambda_k)$
be the Jordan decomposition of $\mathcal{L}$ with respect to  some basis of $\M{d}$. If $e^{t\mathcal{L}}$ is a quantum dynamical semigroup of unital maps on $\M{d}$, 
     \begin{enumerate}[label=(\roman*)]
      \item For every $k$, Re~$(\lambda _k)\leq 0;$
      \item All Jordan blocks corresponding to purely imaginary eigenvalues are one dimensional, that is, if Re~$(\lambda _k)=0$, then $d_k=1;$
      \item If Re~$(\lambda _k)<0$, then  $\lim_{t\to\infty}e^{tJ_{d_k}(\lambda _k)}=0$.
      \end{enumerate}
\end{lem}
\begin{proof} 
 Assume that $\{e^{t\mathcal{L}}\}_{t\geq 0}$ is a quantum dynamical semigroup of unital maps. In particular, it is contractive
 for every $t$. Since any two norms on $\M{d}$ are equivalent, it follows that the matrix entries of $e^{t\mathcal{L}}$ are bounded in the Jordan basis of $\mathcal{L}$ as
 $t$ varies in   $[0, \infty).$ Fix $k$ and consider the $k$-th Jordan block. If $d_k>1$, write $J_{d_k}(\lambda _k)=\lambda I_{d_k}+N_{k}$, where $N_{k}$ is the non-zero nilpotent matrix of order $d_k$. Then,
    \begin{align*}
       e^{tJ_{d_k}(\lambda)}
                         &=e^{t\lambda}I_{d_k}\big(I_{d_k}+tN_k+\frac{t^2N_k^2}{2!}+\cdots +\frac{t^{d_k-1}N_k^{d_k-1}}{(d_k-1)!}\big).
    \end{align*}
Hence
$$[e^{tJ_{d_k}(\lambda _k)}]_{ij}=\left\{\begin{array}{cl}
                               \frac{t^{j-i}}{(j-i)!}~e^{t\lambda _k} & \qquad i\leq j;\\
                               0 & \qquad i>j. 
                            \end{array}\right.$$
Now (i)-(iii) follow easily. 
\end{proof}
 
\begin{lem}\label{lem-ker-algebra} 
    Let $\mathcal{L}:\M{d}\to\M{d}$ be the generator of a quantum dynamical semigroup of unital channels. Then the space $\mcl{P}:=\lspan\{X\in\M{d}:\mathcal{L}(X)=ia X, \quad a\in \mathbb{R}\}$ is a unital $C^*$-sub algebra of $\M{d}$. 
\end{lem}

\begin{proof}
Clearly $\mcl{P}$ contains the identity. As $\mcl{L}$ is a Hermitian preserving map, $\mcl{P}$ is closed under taking adjoints. Now, to prove it is closed under multiplication,  let $X,Y\in\M{d}$ and $a,b\in\mbb{R}$ be such that $\mathcal{L}(X)=iaX$ and $\mathcal{L}(Y)=ibY$. Then  for all $t\geq 0$,
\begin{align*}
    e^{t\mathcal{L}}(X)=e^{ita}X \quad \text{and} \quad e^{t\mathcal{L}}(Y)=e^{itb}Y
\end{align*}
This shows that $X,Y$ is in the peripheral space of $e^{t\mathcal{L}}$. Since the peripheral space of $e^{t\mathcal{L}}$ lies in the multiplicative domain of $e^{t\mathcal{L}}$, we get that
\begin{align*}
e^{t\mathcal{L}}(XY)=e^{t\mathcal{L}}(X)e^{t\mathcal{L}}(Y)=e^{ita+itb}XY, ~~\forall t\geq 0.
\end{align*}
By differentiating both sides at $t=0$, we obtain $\mathcal{L}(XY)=i(a+b)XY$.  This completes the proof.
\end{proof}
The algebra $\mcl{P}$ is known as the {\em peripheral Poisson boundary\/} of $\mcl{L}$, or more appropriately that of the associated quantum dynamical semigroup. For more on this topic see \cite{BD25}.  The following is a one parameter semigroup version of Theorem 3.1 of \cite{BhSa23}. 

\begin{thm}\label{thm-Md-deco-cont} 
    Let $\{\Phi_t=e^{t\mcl{L}}\}_{t\geq 0}$ be a quantum dynamical semigroup of unital  channels on $\M{d}$. Take  $\mcl{P}=\lspan\{X\in\M{d}:\mathcal{L}(X)=ia X, ~\mbox{for some}~ a\in \mathbb{R} \}$. Then the following hold:
    \begin{enumerate}[label=(\roman*)]
        \item $\M{d}=\mcl{P}\oplus\mcl{N}$, where $\mcl{N}=\{Y: \lim_{t\to\infty}\Phi_t(Y)= 0\}$;
        \item  $\Phi _t$ leaves $\mcl{P}$ invariant for every $t$ and restricted to $\mcl{P}$,  $\{\Phi _t\}_{t\geq 0}$ form a semigroup of trace-preserving $*$-automorphisms;
        \item $\Phi _t=\alpha_t+\beta_t$, where $\alpha_t:=\Phi _t\circ\mbf{E}_{\mcl{P}}:\M{d}\to\mcl{P}\subseteq\M{d}$, where  $\mbf{E}_{\mcl{P}}$ is the unique trace-preserving conditional expectation onto $\mcl{P}$ and  $\beta_t:=\Phi _t-\alpha_t:\M{d}\to\mcl{N}\subseteq\M{d}$ satisfies $\lim_{t\to\infty}\beta_t=0$.
    \end{enumerate}
\end{thm}

\begin{proof}
   $(i)$ Using the Jordan decomposition of the generator $\mathcal{L}$ of $\{\Phi _t\}_{t
   geq 0}$, we write $\M{d}$ as a vector space direct sum:
    \begin{align*}
        \M{d}=\ker{\mathcal{L}-\lambda_1\id}^{\nu(\lambda_1)}\oplus \ker{\mathcal{L}-\lambda_2\id}^{\nu(\lambda_2)}\oplus\cdots\oplus \ker{\mathcal{L}-\lambda_n\id}^{\nu(\lambda_n)},
    \end{align*}
    where $\sigma(\mathcal{L})=\{\lambda_1,\lambda_2,\cdots,\lambda_n\}$ is the set of eigenvalues of $\mathcal{L}$ and $\nu(\lambda_j)$ is the algebraic multiplicity of the eigenvalue $\lambda_j$. From Lemma \ref{lem-JCF-gena}, if the real part of $\lambda_j=0$, all the Jordan blocks corresponding to the eigenvalue $\lambda_j$ are of dimension one and hence the algebraic multiplicity is the same as its geometric multiplicity. Therefore,
$$\mcl{P}=\bigoplus_{\lambda\in\sigma_{pi}(\mathcal{L})}\ker{\mathcal{L}-\lambda \id}^{\nu(\lambda)},$$
where $\sigma_{pi}(\mathcal{L})$ be the set of all purely imaginary eigenvalues of $\mcl{L}$. Suppose $\lambda_k\in\sigma(\mcl{L})$ be such that the real part of $\lambda_k\neq 0$, then again by Lemma \ref{lem-JCF-gena}, we have $e^{tJ_{d_k}(\lambda_k)}\to 0$  as $t\to\infty$. Thus, for  $Y\in\ker{\mathcal{L}-\lambda_k \id}^{\nu(\lambda_k)}$, we have $\Phi _t(Y)\to 0$ as $t\to\infty$. By letting $\mcl{N}:=\bigoplus_{\lambda\in\sigma(\mathcal{L})\setminus \sigma_{pi}(\mathcal{L})}\ker{\mathcal{L}-\lambda\id}^{\nu(\lambda)}$, we get the required decomposition. Note that as $\Phi _t$  is contractive in Hilbert-Schmidt norm and $\mcl{P}$ is the span of peripheral eigenvectors, it is reducing for $\Phi _t$. Consequently, $\M{d}=\mcl{P}\oplus\mcl{N}$, is an orthogonal direct sum with respect to the Hilbert-Schmidt inner product.\\
    $(ii)$ Let $X\in\M{d}$ be such that $\mathcal{L}(X)=ia X$, where $a\in \mathbb{R}$. Then, it is easy to see that $\Phi _t(X)=e^{iat}X$ for all $t\geq 0$. This shows that $\mcl{P}$ is in the peripheral space of $\Phi _t$. For quantum channels the peripheral space is a subset of the multiplicative domain.  So we see that the map $\Phi _t\arrowvert_{\mcl{P}}:\mcl{P}\to\mcl{P}$ is a $*$-homomorphism, for every $t\geq 0$.   Since as a linear map $\Phi _t$ has the inverse $e^{-t\mcl{L}}$, on ${\mcl{P}}$, it is injective and due to finite dimensionality it is also a bijection. Therefore restricted
     to $\mcl{P},$ the maps $\{ \Phi _t\}_{t\geq 0}$ form a semigroup of automorphisms.  (This is actually a special case  of a more general result, see Theorem 4.8 of \cite{BD25}.)\\
    $(iii)$ For every $t\geq 0$, let $\alpha_t:=\Phi _t\circ\mbf{E}_{\mcl{P}}$ and $\beta_t:= \Phi _t-\alpha_t$. Given $Z\in\M{d}$, decompose it as $Z=X+Y$ with $X\in\mcl{P}$ and $Y\in\mcl{N}$. Then for all $t\geq 0$,
    $$\beta_t(Z)=\Phi _t(Z)-\alpha_t(Z)=\Phi _t(X)+\Phi _t(Y)-\Phi_t\circ\mbf{E}_{\mcl{P}}(Z)=\Phi _t(Y).$$
    Since $\Phi _t(Y)\to 0$ as $t\to\infty$, we have $\beta_t(\M{d})\subseteq\mcl{N}$ for all $t\geq 0$ and $\norm{\Phi_t-\alpha_t}=\norm{\beta_t}\to 0$ as $t\to\infty$.
\end{proof}


Recall that {\em decoherence-free $*$-subalgebra} of the dynamical semigroup $\{\Phi_t\}_{t\geq 0}$ of unital channels is the largest $*$-algebra $\mcl{D}$ such that for every $X\in\mcl{D}$ we have
\begin{align*}
\Phi_t(X^*X)=\Phi_t(X^*)\Phi_t(X)\qquad\text{and}\qquad \Phi_t(XX^*)=\Phi_t(X)\Phi_t(X^*), \forall t\geq 0.
\end{align*}
Therefore, $\mcl{D}=\bigcap_{t\geq 0}\msc{M}_{\Phi_t}$, where $\msc{M}_{\Phi}$ denotes the multiplicative domain of $\Phi$.
Furthermore, if the generator $\mcl{L}$ is given by \eqref{GKLS}, then from \cite{FR, DFR}, $\mcl{D}$ is equal to the commutator of  the set
\begin{align*}
    \{L_j,L_j^*,[H,L_j],[H,L_j^*],[H,[H,L_j]],[H,[H,L_j^*]],\cdots |1\leq j\leq n\}.
\end{align*}
The results in \cite{FR, DFR} are applicable for quantum dynamical semigroups on $\msc{B}(\mathcal{H})$, where $\mathcal{H}$ is a separable Hilbert space, possibly infinite dimensional. We wish to show that in our context of unital quantum channels the decoherence-free $*$-subalgebra is nothing but the peripheral Poisson boundary and can be described in simpler ways.

\begin{lem}\label{lem-automor-peripheral}
    Let $\{\Phi_t\}_{t\geq 0}$ be a quantum dynamical semigroup of unital channels on $\M{d}$. Then, $\Phi_t(\mcl{D})\subseteq\mcl{D}$ and $\Phi_t\arrowvert_{\mcl{D}}$ is a  $*$-automorphism.
\end{lem}
\begin{proof}
     Suppose $X\in\mcl{D}=\bigcap_{t\geq 0}\msc{M}_{\Phi_t}$ then
\begin{align*}
    \Phi_{t+s}(X^*X)=\Phi_t\big(\Phi_s(X^*X)\big)&=\Phi_t\big(\Phi_s(X^*)\Phi_s(X)\big)\\
    &\geq\Phi_t\big(\Phi_s(X^*)\big)\Phi_t\big(\Phi_s(X)\big)\\
    &=\Phi_{t+s}(X^*)\Phi_{t+s}(X)=\Phi_{t+s}(X^*X),~~~~~~~ \forall s,t\geq 0.
\end{align*}
This shows that $\Phi_t\big(\Phi_s(X^*)\Phi_s(X)\big)=\Phi_t\big(\Phi_s(X^*)\big)\Phi_t\big(\Phi_s(X)\big)$ and consequently, we get $\Phi_s(X)\in\mcl{D}$, for all $s\geq 0$. Since $\mcl{D}\subseteq\msc{M}_{\Phi_s}$,  $\Phi_s$ is a $*$-homomorphism on $\mcl{D}$, for all $s\geq 0$. Now, to prove that it is bijective it is enough to show that $\Phi_s$ is injective. For this, let $X\in\mcl{D}$ such that $\Phi_s(X)=0$ then
\begin{align*}
\ip{X,X}=\tr(X^*X)=\tr\big(\Phi_s(X^*X)\big)=\tr\big(\Phi_s(X^*)\Phi_s(X)\big)=0.
\end{align*}
Hence $X=0.$ This completes the proof.
\end{proof}
\begin{thm}\label{thm-decoh-peri}
Let $\{\Phi_t=e^{t\mcl{L}}\}_{t\geq 0}$ be a quantum dynamical semigroup of unital channels on $\M{d}$ with $\mcl{L}$ as in (\ref{GKLS}). Then, 
the following spaces are equal: \begin{itemize}
\item The peripheral Poisson boundary $\mcl{P}:= \lspan\{X\in\M{d}:\mcl{L}(X)=iaX, \text{ for some }a\in\mbb{R}\}.$
\item  The decoherence-free algebra $\mcl{D}:=\bigcap_{t\geq 0}\msc{M}_{\Phi_t}$.
\item  $\mcl{C}:=\lspan \{X\in \M{d}: i[H, X]=iaX , [L_j, X]=[L_j^*, X]=0, 1\leq j\leq n, \text{ for some } a\in \mbb{R}\}. $
\end{itemize}
\end{thm}
\begin{proof}
By virtue of Theorem \ref{thm-Md-deco-cont}, to show that $\mcl{D}=\mcl{P}$, it is enough to prove that $\mcl{D}^{\perp}=\{Y\in\M{d}: \lim_{t\to\infty}\Phi_t(Y)=0\}$. Suppose $Y\in\M{d}$ satisfies $\lim_{t\to\infty}\Phi_t(Y)=0$. Then for any $X\in\mcl{D}$, any $t\geq 0$, we have
\begin{align*}
\abs{\ip{X,Y}}=\abs{\tr(X^*Y)}=\abs{\tr(\Phi_t(X^*Y))}=\abs{\tr(\Phi_t(X)^*\Phi_t(Y))}\leq d\norm{\Phi_t(X)}\norm{\Phi_t(Y)}\leq d\norm{X} \norm{\Phi_t(Y))}
\end{align*}
This shows that $\ip{X,Y}=0$ and hence $Y\in\mcl{D}^{\perp}$. Now, to prove the reverse inclusion, let $Y\in\mcl{D}^{\perp}$. As the set $\{\Phi_t(Y)\}_{t\geq 0}$ is bounded,  there exists  $\{t_n\}\subset [0,\infty )$, converging to infinity such that $\lim_{n\to\infty}\Phi_{t_n}(Y)=Y_0,$ for some $Y_0\in\M{d}$.
Note that as $\Phi _t$ is contractive in Hilbert-Schmidt norm,
 $\{\norm{\Phi_t(Y)}_2^2\}_{t\geq 0}$ is non-increasing. Consequently $\lim _{n\to \infty}\|\Phi _{s+t_n}(Y)\|_2^2$  $=\|Y_0\|_2^2$ for every $s$. Now, by Kadison-Schwarz inequality, $\Phi_s(Y_0^*Y_0)-\Phi_s(Y_0^*)\Phi_s(Y_0)\geq 0$, and
\begin{align*}
    0=\lim_{n\to\infty}(\norm{\Phi_{t_n}(Y)}_2^2-\norm{\Phi_{s+t_n}(Y)}_2^2)&=\lim_{n\to\infty}\big(\tr(\Phi_{t_n}(Y^*)\Phi_{t_n}(Y))-\tr(\Phi_{s+t_n}(Y^*)\Phi_{s+t_n}(Y))\big)\\
    &=\lim_{n\to\infty}\Big(\tr\big(\Phi_s(\Phi_{t_n}(Y^*)\Phi_{t_n}(Y))\big)-\tr(\Phi_{s+t_n}(Y^*)\Phi_{s+t_n}(Y))\Big)\\
    &=\tr\big(\Phi_s(Y_0^*Y_0)-\Phi_s(Y_0^*)\Phi_s(Y_0)\big).
\end{align*}
So we get $\Phi_s(Y_0^*Y_0)=\Phi_s(Y_0^*)\Phi_s(Y_0)$, for all $s\geq 0$. Similarly, one can show that  $\Phi_s(Y_0Y_0^*)=\Phi_s(Y_0)\Phi_s(Y_0^*)$, for all $s\geq 0$. Therefore, $Y_0\in\mcl{D}$. Now, we will also show that $Y_0\in \mcl{D}^{\perp}$. To see this first observe that from Lemma \ref{lem-automor-peripheral}, $\Phi_t$ restricted to $\mcl{D}$ is bijective. Given $t\geq 0$ and $Z\in\mcl{D}$, let $X\in\mcl{D}$ be the unique element such that $\Phi_t(X)=Z$, then
\begin{align*}
    \ip{Z,\Phi_t(Y)}=\ip{\Phi_t(X),\Phi_t(Y)}=\tr(\Phi_t(X^*)\Phi_t(Y))=\tr(\Phi_t(X^*Y))=\ip{X,Y}=0. 
\end{align*}
This shows that $\Phi_t(Y)\in \mcl{D}^{\perp}$ for every $t\geq 0$ and hence $Y_0=\lim_{n\to\infty}\Phi_{t_n}(Y)\in\mcl{D}^{\perp}$. Therefore $Y_0=0$. Now, since $\{\norm{\Phi_t}_2\}_{t\geq 0}$ is non-increasing and $\lim_{n\to\infty}\Phi_{t_n}(Y)=0$, we conclude that $\lim_{t\to\infty}\Phi_t(Y)=0$. This completes the proof of $\mcl{D}=\mcl{P}$.

Observe that for any $Z\in\M{d}$, by direct computation:
\begin{align}\label{eq-Gkls-dec}
    \mcl{L}(Z^*Z)=Z^*\mcl{L}(Z)+\mcl{L}(Z^*)Z+\sum_{j=1}^n[L_j,Z]^*[L_j,Z].
\end{align}
    Suppose $X\in\mcl{P}$ is such that  $\mcl{L}(X)=iaX$ with $a\in \mathbb{R}$. Then $\Phi _t(X)=e^{ita}X $. Since we have assumed that $\Phi_t$ is trace preserving, by Kadison-Schwarz inequality, we conclude that $\Phi _t(X^*X)=X^*X$ and $\Phi _t(XX^*)=XX^*$, for all $t\geq 0$. By differentiation at $t=0$, we get $\mcl{L}(X^*X)=0=\mcl{L}(XX^*)$. Thus, by letting $Z=X$ in  equation \eqref{eq-Gkls-dec}, we get $\sum_{j=1}^n[L_j,X]^*[L_j,X]=0$, and hence $[L_j,X]=0$ for every $j$. Similarly by taking taking $Z=X^*$, we get $[L_j,X^*]=0$ for every $j$. Hence, 
$iaX=\mcl{L}(X)=i[H,X]$, proving $\mcl{P}\subseteq \mcl{C}.$ The reverse inclusion is clear from 
the formula (\ref{GKLS}) of $\mcl{L}$.
\end{proof}
The following theorem is the continuous analogue of Theorem \ref{asymptotic}. The conditional automorphism under consideration may not be unique. It depends upon the self-adjoint operator $H$ chosen in (\ref{GKLS}).
\begin{thm}\label{thm-asymptotic-conti}
    Let $\{\Phi_t=e^{t\mcl{L}}\}$ be a dynamical semigroup of unital channels on $\M{d}$ with the generator $\mcl{L}$  given by  \eqref{GKLS}. That is, 
\begin{equation*}
\mathcal{L}(X)=i[H,X]-\frac{1}{2}\sum _{j=1}^n(L_j^*L_jX+XL_j^*L_j-2L_j^*XL_j),~~\forall X\in \M{d},
\end{equation*}
 for some $H, L_j, 1\leq j\leq n$ in $\M{d}, ~n\geq 0,$ satisfying $H=H^*, ~\sum _{j=1}^nL_jL_j^*=\sum _{j=1}^nL_j^*L_j.$ 
    Then
    \begin{align*}
        \lim_{t\to\infty}\norm{\Phi_t- \mathrm{Ad}_{U_t}\circ \mbf{E}_{\mcl{P}}}=0,
    \end{align*}
    where $U_t=e^{-itH}, ~~\forall t\geq 0$ and $\mbf{E}_{\mcl{P}}$ is the conditional expectation map onto the peripheral Poisson boundary $\mcl{P}$.
\end{thm}
\begin{proof}  
 Suppose 
$X$ is such that $i[H,X]=iaX$ for some $a\in \mbb{R}$ and $[L_j, X]=[L_j^*,X]=0$ for all $j$. Then clearly,
$\mcl{L}(X)=iaX.$
Consequently, by taking the power series expansion of the exponential, 
we get  $e^{ita}X=e^{t\mcl{L}}(X)= \mathrm{Ad}_{U_t}(X)$. Making use of Theorem \ref{thm-decoh-peri},
we conclude that $\Phi _t(X)= \mathrm{Ad}_{U_t}(X)$, for all $X\in\mcl{P}$. Then for any $Z=X+Y$ with $X\in\mcl{P}$ and $Y\in\mcl{P}^\perp$, by Theorem \ref{thm-Md-deco-cont},
we have
    \begin{align*}
        \big(\Phi_t-\mathrm{Ad}_{U_t}\circ\mbf{E}_{\mcl{P}}\big)(Z)=\Phi_t(X)+\Phi_t(Y)-\mathrm{Ad}_{U_t}(X)=\Phi_t(Y)\to 0, ~~~ \mbox{as } t\to\infty.
    \end{align*}
\end{proof}
\begin{cor}
     Let $\{\Phi_t=e^{t\mcl{L}}\}$ be a dynamical semigroup of unital channels on $\M{d}$ with the generator $\mcl{L}$ as in  \eqref{GKLS}.  If $H$ is a scalar matrix, then $\Phi _t$ converges to the conditional expectation $\mbf{E}_{\mcl{P}}$ as $t$ goes to infinity.
\end{cor}
\begin{proof}
    Assume that $H$ is a scalar matrix and $U_t=e^{-itH}$, then $\mathrm{Ad}_{U_t}(X)=X$, for all $X\in\M{d}$ and thus from Theorem \ref{thm-asymptotic-conti}, $\Phi_t$ converges to a conditional expectation onto the peripheral space $\mcl{P}$.
\end{proof}
The converse of this result is not true as the following remark shows.
\begin{rmk}
A dynamical semigroup of unital channels  $\{\Phi_t\}_{t\geq 0}$ converges to $\mbf{E}_{\mcl{P}}$, as $t\to \infty$ does not imply that  $H$ is a scalar matrix. For example, take  $\mcl{L}(X)=i[H,X]+\frac{1}{2}(2L^*XL-LL^*X-XLL^*)$, where $L$ is a normal matrix such that $L+L^*$ is not a scalar matrix and $H=L+L^*$. Note that, if $X\in\mcl{P}$, using the same argument as in Theorem \ref{thm-asymptotic-conti}, we conclude that $[L_j,X]=0=[L_j^*,X]$ and consequently  $\mcl{L}(X)=i[H,X]=i[L_j+L_j^*,X]=0$.
Now, by taking the power series, we get $\Phi_t(X)=X$ for all $t\geq 0$. By writing any $Z=X+Y$ with $X\in\mcl{P}$ and $Y\in\mcl{N}$, we have
 \begin{align*}
\Phi_t(Z)=\Phi_t(X)+\Phi_t(Y)=X+\Phi_t(Y)=\mbf{E}_{\mcl{P}}(Z)+\Phi_t(Y)\to\mbf{E}_{\mcl{P}}(Z),~~~~\mbox{as } t\to\infty.
 \end{align*}
\end{rmk}

\begin{thm}\label{thm-semigp-eve-MU}
    Every quantum dynamical semigroup of unital channels on $\M{d}$ is eventually mixed unitary.
\end{thm}

\begin{proof}
    Let $\{\Phi_t\}_{t\geq 0}$ be a quantum dynamical semigroup of unital  channels on $\M{d}$.  We use the notation and results of  Theorem \ref{thm-Md-deco-cont} and Theorem \ref{thm-asymptotic-conti}. From the proof of Theorem \ref{thm-asymptotic-conti}, we know that 
    $$\Phi _t\arrowvert_{\mcl{P}}=\mathrm{Ad}_{U_t}\arrowvert_{\mcl{P}},$$
    where $U_t=e^{-itH}$.
 Take $\tau_t:=\mathrm{Ad}_{U_t^*}\circ \Phi_t$, then
 \begin{align*}
     \tau_t(X)&=\mathrm{Ad}_{U_t^*}\circ \Phi _t(X)
     =\mathrm{Ad}_{U_t^*}\circ \mathrm{Ad}_{U_t}(X)=X,~~~~\forall X\in\mcl{P}.
 \end{align*}
 From Theorem \ref{thm-Md-deco-cont}, we also know that $\Phi_t=\alpha_t+\beta_t$,  where $\alpha_t=\Phi_t\circ\mbf{E}_{\mcl{P}}$ and  $\lim_{t\to\infty}\beta_t=0$. Therefore, $\mathrm{Ad}_{U_t^*}\circ\alpha_t=\mbf{E}_{\mcl{P}}$ for all $t\geq 0$, and 
\begin{align*}
    \norm{\beta_t}&=\norm{\Phi _t-\alpha_t}=\norm{\mathrm{Ad}_{U_t^*}\circ(\Phi _t-\alpha_t)}=\norm{\tau_t-\mbf{E}_{\mcl{P}}}.
\end{align*}
By choosing $r_{\mcl{P}}>0$  as in Theorem \ref{thm-cond-exp-mu}, there exists a $t_0\geq 0$ such that $\norm{\tau_t-\mbf{E}_{\mcl{P}}}\leq r_{\mcl{P}}$, for all $t\geq t_0$. Now, by Theorem \ref{thm-cond-exp-mu}, it follows that $\tau_t$ is mixed unitary for all $t\geq t_0$ and consequently, $\Phi _t=\mathrm{Ad}_{U_t}\circ\tau_t$ is mixed unitary for all $t\geq t_0$.
\end{proof}

We now briefly return to the case of discrete semigroups. The proof of the following Remark is very similar to the proof of  Theorem \ref{thm-decoh-peri} and so we omit it.
\begin{rmk}
Let $\Phi$ be a unital channel on $\M{d}$. Consider the discrete semigroup of unital channels $\{\Phi^n: n\geq 1\}$, then its
decoherence free algebra is same as the peripheral Poisson boundary. That is, $\bigcap_{n\geq 1}\msc{M}_{\Phi^n}=\lspan\{X:\Phi(X)=\lambda X,~\abs{\lambda}=1\}.$ 
\end{rmk}

\begin{defn} \cite{HRS20}
 For any unital trace-preserving linear map $\Phi:\M{d}\to\M{d}$,  the \emph{mixed unitary time } or \emph {mixed unitary index } of $\Phi$ is defined as the least $n\in\mbb{N}$ such that $\Phi^k$ is mixed unitary for all $k\geq n$.  
\end{defn}

From Theorem \ref{thm-even-mu-discrete}, it follows that the mixed unitary index of a unital quantum channel is finite. One can ask whether the mixed unitary index is independent of $\Phi$. i.e., Is there an integer $n$ such that for every unital quantum channel $\Phi$ on $\M{d}$,  $\Phi^k$ is mixed unitary, $\forall k\geq n$.  The answer is negative as the following result shows. 

\begin{thm}\label{thm-No-univ-MUindex}
    Let $d\geq 3$, then there is no $n$ such that $\Phi^n$ is mixed unitary for all unital quantum channels $\Phi:\M{d}\to\M{d}$. 
\end{thm}

\begin{proof}
    Suppose that there exists $n\geq 2$ such that $\Phi^n$ is mixed unitary for all unital quantum channels $\Phi$ on $\M{d}$. Let  $\mcl{L}$ be the generator of a quantum dynamical semigroup of unital channels, which is not mixed unitary for some $t_0>0$. (Note that such $\mcl{L}$ exists. For example, take $\mcl{L}$ as in  Lemma \ref{eg-semigrp-not-mixed unitary}.) For $n\geq 1$, take $\Phi _n:=e^{\frac{t_0}{n}\mcl{L}}$. Then $\Phi _n$ is a unital quantum channel on $\M{d}$, but $(\Phi _n)^n=e^{t_0\mcl{L}}$ is not mixed unitary. 
\end{proof}


\section{Behavior of dynamical semigroup of quantum channels near origin}
Given any dynamical semigroup of unital quantum channels on $\M{d}$, we have shown that (See Theorem \ref{thm-semigp-eve-MU}) it is eventually mixed unitary. In this section, we examine the behavior of such semigroups near $t=0$.  We do this by first giving a new characterization of generators of semigroups having only mixed unitary channels.

\textbf{Notations:}
\begin{align}
    \msc{W}(d):&=\{\Phi\in\msc{H}(d):\Phi \text{ is unital and trace-preserving}\}.\\
   \msc{W}_1(d):&=\{\Phi\in \msc{W}(d):\norm{\Phi}=1\}.\\
   \mcl{C}(d):&=~\mbox{The cone generated by mixed unitary channels on $\M{d}$}.\\
    \msc{D}(d):&=\{\Gamma\in\msc{W}(d): \ip{\id,\Gamma}=0, \ip{\mathrm{Ad}_U,\Gamma}\geq 0,~~\forall U\in\mbb{U}(d)\}
\end{align}
In other words, $\msc{W}(d)$ is the set of unital,  trace-preserving Hermitian preserving maps on $\M{d}$ and $\msc{W}_1(d)$ is the collection of unit elements in it. Further,  $\msc{D}(d)$ is a convenient subset of the dual of the cone generated by mixed unitary channels. Note that the inner product here is defined using Choi matrices as in \eqref{inner product}. This will be used in the characterization below in Theorem \ref{new characterization}.

Consider the set $\mcl{C}(d)^\circ \cap\msc{W}_1(d)$. Then, it is a non-empty compact set. For, since the completely depolarizing channel $\delta_d\in\msc{W}_1(d)$ and $\ip{\Phi,\delta_d}=\frac{1}{d^2}\tr(\msc{J}(\Phi))\geq 0$, for all $\Phi\in\mcl{C}(d)$, we have $\delta_d\in\mcl{C}(d)^\circ \cap\msc{W}_1(d)$. It is compact because it is a closed subset of the compact set $\msc{W}_1(d)$.
\begin{lem}\label{lem-mixed unitary-semigroup}
Suppose $\mcl{L}\in\msc{H}(d)$ with $\ip{\Gamma, \mcl{L}}\geq 0$ for all $\Gamma\in\msc{D}(d)$. Then, for every $r>0$ there exists $\eta >0$ such that if  $\Gamma\in\mcl{C}(d)^{\circ}\cap\msc{W}_1(d)$ satisfies $\ip{\Gamma,\id}<\eta $ then $\ip{\Gamma,\mcl{L}}\geq -r$.
\end{lem}
\begin{proof}
We will prove by contradiction. Suppose there exists a $r>0$  and a sequence  $\{\Gamma_n \}$ $\mcl{C}(d)^\circ \cap\msc{W}_1(d)$ with $\ip{\Gamma_n,\id}<\frac{1}{n}$ such that 
    \begin{align}\label{eq-conditionally}
        \ip{\Gamma_n,\mcl{L}}<-r,
    \end{align}
 for every $n.$   Since $\mcl{C}(d)^\circ \cap\msc{W}_1(d)$ is compact, there exist a $\Gamma\in\mcl{C}(d)^\circ \cap\msc{W}_1(d)$ and a sub sequence $\{\Gamma_{n_k}\}$ such that $\lim_{n_k}\Gamma_{n_k}=\Gamma$ and
    \begin{align*}
        0\leq \ip{\Gamma,\id}=\lim_{n_k}\ip{\Gamma_{n_k},\id}\leq\lim_{n_k}\frac{1}{n_k}= 0.
    \end{align*}
   This shows that $\Gamma\in\msc{D}(d)$. Now, from equation \eqref{eq-conditionally}, we have $\ip{\Gamma, \mcl{L}}\leq -r<0$. This contradicts the assumption on $\mcl{L}$. 
\end{proof}
Recall that K{\"u}mmerer et al.\ \cite{KuMa87} characterizes  (See Theorem \ref{thm-mu-semi-gene}) the dynamical semigroups that are mixed unitary channels for all $t\geq 0$.  Here, we provide another necessary and sufficient condition for a semigroup to be mixed unitary.

\begin{thm}\label{thm-schonberg-mixed unitary}\label{new characterization}
     Let $\mcl{L}\in\msc{H}(d)$ be such that $\mcl{L}(I_d)=0$ and $\tr(\mcl{L}(X))=0, \forall X\in\M{d}$. Then, the following are equivalent:
     \begin{enumerate}[label=(\roman*)]
         \item $e^{t\mcl{L}}$ is mixed unitary for all $t\geq 0$;
         \item The generator $\mcl{L}$ is in the closed cone generated by $\{\mathrm{Ad}_U-\id:U\in\mbb{U}(d)\}$; 
         \item $\ip{\Gamma,\mcl{L}}\geq 0$ for all $\Gamma\in\msc{D}(d)$.
     \end{enumerate}
 \end{thm}
 \begin{proof}
$(i)\Rightarrow (ii)$ 
Assume that $e^{t\mcl{L}}$ is mixed unitary for all $t$. That is, $e^{t\mcl{L}}=\sum_{j=1}^{n(t)}\lambda_j(t)\mathrm{Ad}_{U_j(t)}$,
for some unitaries $U_j(t)$ and scalars $\lambda_j(t)\in[0,1]$ with $\sum_{j=1}^{n(t)}\lambda_j(t)=1$.  Therefore,
\begin{align*}
    \frac{e^{t\mcl{L}}-\id}{t}=\sum_{j=1}^{n(t)}\lambda_j(t)(\mathrm{Ad}_{U_j(t)}-\id)\in \text{cone}\{\mathrm{Ad}_U-\id:U\in\mbb{U}(d)\},
\end{align*} 
Consequently, $\mcl{L}=\lim_{t}\frac{e^{t\mcl{L}}-\id}{t}$ is in the closed cone generated by $\{\mathrm{Ad}_U-\id:U\in\mbb{U}(d)\}$.\\
$(ii)\Rightarrow (iii)$ Assume that $\mcl{L}$ is in the closed cone generated by  $\{\mathrm{Ad}_U-\id:U\in\mbb{U}(d)\}$. Let $\Gamma\in\msc{D}(d)$ be arbitrary. Then, for any unitary $U\in\mbb{U}(d)$,
\begin{align*}
   0\leq \ip{\Gamma,\mathrm{Ad}_U}=\ip{\Gamma,\mathrm{Ad}_U}-\ip{\Gamma,\id}=\ip{\Gamma,\mathrm{Ad}_U-\id}.
\end{align*}
Therefore, for all $\Phi\in\text{cone}\{\mathrm{Ad}_U-\id:U\in\mbb{U}(d)\}$ we have $\ip{\Gamma,\Phi}\geq 0$. Consequently, we also get $\ip{\Gamma,\Phi}\geq 0$, for all $\Phi$ is in the closed cone generated by $\{\mathrm{Ad}_U-\id:U\in\mbb{U}(d)\}$. In particular, we get $\ip{\Gamma,\mcl{L}}\geq 0$.\\
 $(iii)\Rightarrow (i)$  To prove $e^{t\mcl{L}}\in\msc{MU}(d)$ for all $t\geq 0$, it is enough to show that $e^{\mcl{L}}\in\msc{MU}(d)$. 
     This can be further simplified by showing that  $(\id+\frac{\mcl{L}}{n})\in\msc{MU}(d), ~~ \forall n\geq n_0$, for some fixed $n_0$, because $\msc{MU}(d)$ is closed under compositions and $e^{\mcl{L}}=\lim_{n\to\infty}(\id+\frac{\mcl{L}}{n})^n$.
    
    Given $\epsilon >0$, let $\mcl{L}_{\epsilon}:=\mcl{L}+\epsilon \delta_d$, where $\delta_d$ is the completely depolarizing channel. Set $r=\frac{\epsilon}{d^2}$. As $\ip{\Gamma, \mcl{L}}\geq 0$ for all $\Gamma\in\msc{D}(d)$. From Lemma \ref{lem-mixed unitary-semigroup}, there exists a $\eta >0$ such that  $\ip{\Gamma,\mcl{L}}\geq -r$ for all $\Gamma\in\mcl{C}^{\circ}\cap\msc{W}_1(d)$ with $\ip{\Gamma,\id}<\eta $, where $\mcl{C}$ is the closed cone generated by $\{\mathrm{Ad}_U:U\in\mbb{U}(d)\}$. 
    Let $\Gamma\in\mcl{C}^\circ\cap\msc{W}_1(d)$ be arbitrary,\\
    \ul{Case (i):} If $0\leq \ip{\Gamma,\id}<\eta $, then
    \begin{align*}
       \ip{\Gamma,\id+\frac{\mcl{L}_\epsilon}{n}}&=\ip{\Gamma,\id}+\frac{1}{n}\big(\ip{\Gamma,\mcl{L}}+\epsilon\ip{\Gamma,\delta_d}\big)
                                                      \geq\frac{1}{n}\big(\ip{\Gamma,\mcl{L}}+\epsilon\ip{\Gamma,\delta_d}\big)
                                                      \geq\frac{1}{n}(-\frac{\epsilon}{d^2}+\epsilon\frac{1}{d^2})
                                                      =0.
    \end{align*}\\
    \ul{Case (ii):} If $\ip{\Gamma,\id}\geq\eta $, then
    \begin{align*}
        \ip{\Gamma,\id+\frac{\mcl{L}_\epsilon}{n}}&=\ip{\Gamma,\id}+\frac{1}{n}\big(\ip{\Gamma,\mcl{L}}+\epsilon\ip{\Gamma,\delta_d}\big)\\
                                                        &\geq \eta+\frac{1}{n}\big(\ip{\Gamma,\mcl{L}}+\epsilon\ip{\Gamma,\delta_d}\big)\\
                                                        &\geq \eta+\frac{1}{n}\ip{\Gamma,\mcl{L}}\\
                                                        &\geq \eta-\frac{1}{n}\norm{\mcl{L}}_2.
    \end{align*}
     Here $\norm{\mcl{L}}_2^2=\ip{\mcl{L},\mcl{L}}$. Take $n_0\geq 1$ such that $n_0\eta\geq\norm{\mcl{L}}_2$, then
     \begin{align}
        \ip{\Gamma,\id+\frac{\mcl{L}_\epsilon}{n}}=\eta-\frac{1}{n}\norm{\mcl{L}}_2\geq \eta-\frac{1}{n_0}\norm{\mcl{L}}_2\geq 0, ~~~\forall n\geq n_0.
     \end{align}
     Therefore, for every $n\geq n_0$ and $ \Gamma\in\mcl{C}^\circ\cap\msc{W}_1(d)$,
     \begin{align*}
         \ip{\Gamma,\id+\frac{\mcl{L}}{n}}=\lim_{\epsilon\to 0}\ip{\Gamma,\id+\frac{\mcl{L}_\epsilon}{n}}\geq 0.
     \end{align*}
     In particular, $\ip{\Gamma,\id+\frac{\mcl{L}}{n}}\geq 0$ for all $\Gamma\in\msc{MU}(d)^\circ\cap\msc{W}(d)$ and $n\geq n_0$. Now, from Lemma \ref{lem-interior-dual} (iii), it follows that $\id+\frac{\mcl{L}}{n}$ is mixed unitary for all $n\geq n_0$. This completes the proof.\end{proof}
\begin{rmk}
  From the above Theorem \ref{thm-schonberg-mixed unitary}, it follows that if the infinitesimal generator  $\mcl{L}$ has the form $\Psi-\id$, where $\Psi$ is a mixed unitary channel, then the semigroup $e^{t\mcl{L}}$ is mixed unitary for all $t\geq 0$. Interestingly, there  exist semigroups of mixed unitary channels whose infinitesimal generator is not of this form.
  As an example, consider the  linear map $\mcl{L}:\M{3}\to\M{3}$ given by
    \begin{align*}
        \mcl{L}(X)=\frac{1}{2}(\tr(X)I_3-X^{\T})-X, ~~~~\forall X\in\M{3}.
    \end{align*}
    Since the adjiont of the map $X\mapsto \frac{1}{2}(\tr(X)I_3-X^{\T})$ is itself, there exists Hermitian matrices $A_j\in\M{3}$ such that $\sum_{j}A_j^2=I_3$ and 
\begin{align*}
    \mcl{L}(X)=\sum_{j=1}^\ell A_jXA_j-X=\sum_{j=1}^\ell \big(A_jXA_j-\frac{1}{2}(A_j^2X+XA_j^2)\big), ~~\forall X\in\M{3}
\end{align*}
 By Theorem \ref{thm-mu-semi-gene}, the semigroup $e^{t\mcl{L}}$ is mixed unitary for all $t\geq 0$. 
\end{rmk}

\begin{thm}\label{thm-semigroup-not mu-intially}
    Let $\{\Phi _t\}_{t\geq 0}$ be a quantum dynamical semigroup of unital channels. Then the set $\{t\in [0, \infty ): 
    \Phi _t \text{ is not mixed unitary}\}$ is open in $\mathbb{R}$. Moreover one of the following happens:
    \begin{enumerate}[label=(\roman*)]
        \item $\Phi _t$ is mixed unitary for all $t\geq 0$;
        \item There exist $0<t_0\leq t_1<\infty $ such that
        $\Phi _t$ is not mixed unitary for $t\in (0, t_0)$ and is mixed unitary for $t=t_0$ and for $t\in [t_1, \infty )$.
    \end{enumerate}
\end{thm}
\begin{proof} Take $N=\{t\in [0, \infty ): 
    \Phi _t \text{ is not mixed unitary}\}$.  Since the set of mixed unitary channels is closed it is obvious that $N$ is open.
  Let $\mcl{L}$ be the the generator of 
$\{\Phi _t\}$ so that $\Phi _t=e^{t\mcl{L}}.$
    Suppose that $\Phi _s$ is not mixed unitary for some $s>0$,  from Theorem \ref{thm-schonberg-mixed unitary}, there exists a $\Gamma\in\msc{D}(d)$ such that $\ip{\Gamma,\mcl{L}}<0$.  Note that if two matrices are Hermitian, then the trace of their product  is a real number. This allows us to  define a real valued function $g$ on $[0,\infty)$ by
 $  g(t):=\ip{\Gamma,e^{t\mcl{L}}},~~t\in [0, \infty).$
Then $g$ is  a differentiable function with $g(0)=0$. Furthermore, 
    \begin{align*}
        g'(0) =\lim_{t\to 0}\frac{\ip{\Gamma,e^{t\mcl{L}}}-\ip{\Gamma,\id}}{t}
             =\ip{\Gamma, \mcl{L}}<0.
    \end{align*}
    Therefore, there exists a $s_0>0$ such that $g(t)<0$ for all $t\in(0,s_0)$. i.e., $\ip{\Gamma, \Phi _t}<0$ for all $0<t<s_0$. Now, by Lemma \ref{lem-interior-dual} we get that $\Phi _t$ is not mixed unitary for all $t\in (0,s_0)$.
 
  Let $t_0:=\max\{t: \Phi _s \mbox{ is not mixed unitary }, \forall s\in (0,t)\}$.  We claim that $\Phi _{t_0}$ is a mixed unitary. In deed this clear from openness of $N$, as otherwise $t_0\in N$ and this would mean $[t_0, t_0+\epsilon)\subseteq N$ for some $\epsilon >0.$ Now, define $$t_1:=\min\{t: \Phi _s \mbox{ is mixed unitary }, \forall s\in [t,\infty)\}.$$ By Theorem \ref{thm-semigp-eve-MU}, $t_1$ is finite.
Clearly $t_0\leq t_1$. 
\end{proof}

We wish to remark that in the previous theorem, we do not know whether one can take $t_1=t_0$ or not. Considering the general features of numerous  situations of eventual positivity (See for example \cite{Ja17})  we would guess that there could be examples for which $t_0<t_1$ and  the set of points $t$ in $[t_0, t_1]$ for which $\Phi _t$ is a mixed unitary channel is a discrete set. We are not able to confirm this in any of the following examples.

We now give an example of a dynamical semigroup of unital quantum channels, which is not mixed unitary for some  $t>0$. For any $B\in\M{d}$, let $\mu(B)$ denote the set of all singular values of $B$.
 
\begin{lem}\cite[Corollory 15]{MeWolf09}\label{lem-min-attaining}
    Given non-negative numbers $\mu_1\geq \mu_2\geq \cdots \geq \mu_d$ we have
    \begin{align}
        \min_{A}\{\tr(A\ol{A}): \mu(A)=\{\mu_j\}_{j=1}^d\}&= \begin{cases*}
      -2\Big(\sum_{j=1}^{\frac{d}{2}}\mu_{2j}\mu_{2j-1}\Big) & if $d$ is even \\
     -2\Big(\sum_{j=1}^{\frac{d-1}{2}}\mu_{2j}\mu_{2j-1}\Big)+\mu_d^2 & if $d$ is odd
    \end{cases*}
    \end{align}
\end{lem}

 \begin{lem}\label{lem-singular-unitary}
     Let $F=\sum_{i,j=1}^d E_{ij}\otimes E_{ji}\in\M{d}\otimes\M{d}$. For any $B\in\M{d}$, let $W=(I_d\otimes B)F(I_d\otimes B)^*$. Then 
     \begin{align}
              \min_{A}\{\tr(A\ol{A}): \mu(A)=\mu(B)\}&\leq\min\{\tr(W\ranko{(I_d\otimes U) e}{(I_d\otimes U)e}): U\in\mbb{U}(d)\},
     \end{align}
     where $e=\sum_{j=1}^d e_j\otimes e_j$.
 \end{lem}
 \begin{proof}
     For every unitary $U\in\mbb{U}(d)$ we have
     \begin{align*}
         \tr(W\ranko{(I_d\otimes U) e}{(I_d\otimes U)e})&=\ip{(I_d\otimes U)e, W((I_d\otimes U)e)}
                                                         =\ip{\sum_je_j\otimes Ue_j,W(\sum_k e_k\otimes Ue_k)}\\
                                                         &=\sum_{j,k}\ip{e_j\otimes B^*Ue_j,F( e_k\otimes B^*Ue_k)}
                                                         =\sum_{j,k}\ip{e_j\otimes B^*Ue_j, B^*Ue_k\otimes e_k}\\
                                                         &=\sum_{j,k}\ip{e_j,B^*Ue_k}\ip{B^*Ue_j,e_k}
                                                         =\tr(B^*U\ol{B^*U}).
     \end{align*}
Note that if $A=B^*U$, then $A$ has the same singular values as $B$. Therefore, we have $ \min_{A}\{\tr(A\ol{A}): \mu(A)=\mu(B)\}\leq \tr(W\ranko{(I_d\otimes U)e}{(I_d\otimes U)e})$ for all $U\in\mbb{U}(d)$. Consequently,
$$ \min_{A}\{\tr(A\ol{A}): \mu(A)=\mu(B)\}\leq\min\{\tr(W\ranko{(I_d\otimes U) e}{(I_d\otimes U)e}): U\in\mbb{U}(d)\}.$$
 \end{proof}

\begin{lem}\label{lem-example-not mixed unitary semigroup}
 Let $B\in\mbb{U}(3)$ be any unitary matrix such that $\tr(B^*B^{\T})=-1$, define the linear maps  $\mcl{L},\Gamma$ on $\M{3}$ by 
\begin{align*}
    \mcl{L}(X):&=\frac{1}{2}\big(\tr(X)I_3-BX^{\T}B^*-2X\big),\\
    \Gamma(X):&=\frac{1}{2}\Big(BX^{\T}B^*+\frac{1}{3}\tr(X)I_3\Big),~~~~\forall X\in\M{3}.  
\end{align*}
Then the following holds:
\begin{enumerate}[label=(\roman*)]
    \item $\mcl{L}$ is Hermitian preserving map with $\mcl{L}(I_3)=0$ and $\tr(\mcl{L}(X))=0$.
    \item $\Gamma\in\msc{D}(3)$.
    \item $\ip{\Gamma, \mcl{L}}=-\frac{1}{9}$.
\end{enumerate}
\end{lem} 

\begin{proof}
    $(i)$ Follows from the definition of $\mcl{L}$.\\
    $(ii)$ Clearly, $\Gamma$ is unital and trace-preserving map. Take $W=(I_d\otimes B)F(I_d\otimes B)^*$, where $F=\sum_{i,j=1}^3E_{ij}\otimes E_{ji}$. Then $\msc{J}(\Gamma)=\frac{1}{6}(W+\frac{I_3\otimes I_3}{3})$, and 
     \begin{align*}
       \ip{\mathrm{Ad}_{I_3},\Gamma}=\tr(\msc{J}(\Gamma)\ranko{\Omega_3}{\Omega_3})&=\frac{1}{6}\tr\big((W+\frac{I_3\otimes I_3}{3})\ranko{\Omega_3}{\Omega_3}\big)\\
                                    &=\frac{1}{6}\big(\tr(W\ranko{\Omega_3}{\Omega_3})+\frac{1}{3}\big)\\
                                    &=\frac{1}{6}\big(\frac{1}{3}\tr(B^*B^{\T})+\frac{1}{3}\big)=0. &(\because \tr(B^*B^{\T})=-1).
    \end{align*}
    Now, for any unitary $U\in\mbb{U}(3)$, by using the Lemma \ref{lem-singular-unitary}  and \ref{lem-min-attaining}, we have
    \begin{align*}
        \ip{\mathrm{Ad}_{U},\Gamma}=\tr(\msc{J}(\mathrm{Ad}_U)\msc{J}(\Gamma))
                                   &=\tr\Big(\msc{J}(\mathrm{Ad}_U)\frac{1}{6}\big(W+\frac{1}{3}I_3\otimes I_3\big)\Big)\\
                                   &=\frac{1}{6}\tr(\msc{J}(\mathrm{Ad}_U)W)+\frac{1}{18}\tr(\msc{J}(\mathrm{Ad}_U))\\
                                   &=\frac{1}{18}\tr(W\ranko{(I_3\otimes U^*) e}{(I_3\otimes U^*)e})+\frac{1}{18}\notag\\
                               &\geq \frac{1}{18} \min\{\tr(W\ranko{(I_3\otimes V) e}{(I_3\otimes V) e}): V\in\mbb{U}(3)\}+\frac{1}{18}\notag\\
                               &\geq\frac{1}{18}\min\{\tr(A\ol{A}): \mu(A)=\mu(B)\}+\frac{1}{18}  \notag\\
                               &=\frac{1}{18}\min\{\tr(A\ol{A}): \mu(A)=(1,1,1)\}+\frac{1}{18} ~~~~(\because \mu(B)=(1,1,1)) &\notag\\
                               &= \frac{1}{18} (-2+1)+\frac{1}{18}
                                   =0.
    \end{align*}
     This shows that $\Gamma\in\msc{D}(3)$.\\
    $(iii)$ Let $W=(I_d\otimes B)F(I_d\otimes B)^*$, where $F=\sum_{i,j=1}^3 E_{ij}\otimes E_{ji}$. Let $\eta:\M{3}\to\M{3}$ be a linear map given by
     \begin{align*}
         \eta(X):=\frac{1}{2}(\tr(X)I_3-X^{\T}), ~~~~~~\forall X\in\M{3}.
     \end{align*}
    Then it is easy to see that $\mcl{L}=\mathrm{Ad}_{B^*}\circ \eta -\id$ and
    \begin{align*}
        \msc{J}(\mathrm{Ad}_{B^*}\circ \eta)&=(I_3\otimes B)\msc{J}(\eta)(I_3\otimes B^*)=\frac{1}{6}(I_3\otimes B)(I_3\otimes I_3-F)(I_3\otimes B^*)=\frac{1}{6}(I_3\otimes I_3-W),\\
    \msc{J}(\Gamma)&=\frac{1}{6}(\frac{1}{3}I_3\otimes I_3+W).
    \end{align*}  
    Therefore, 
    \begin{align*}
        \ip{\Gamma, \mcl{L}}=\ip{\Gamma,\mathrm{Ad}_{B^*}\circ \eta}-\ip{\Gamma,\mathrm{Ad}_{I_3}}
                      &=\tr(\msc{J}(\Gamma)\msc{J}(\mathrm{Ad}_{B^*}\circ \eta))-0\\
                      &=\frac{1}{36}\tr\big((W+\frac{1}{3}I_3\otimes I_3)(I_3\otimes I_3-W)\big)\\
                      &=\frac{1}{36}\tr\big(W-W^2+\frac{1}{3}(I_3\otimes I_3-W)\big)\\
                      &=\frac{1}{36}\tr(W-I_3\otimes I_3+\frac{I_3\otimes I_3}{3}-\frac{W}{3})& (\because W^2=I_3\otimes I_3)\\
                      &=\frac{1}{36}\frac{2}{3}\tr(W-I_3\otimes I_3)\\
                      &=\frac{1}{54}\big(\tr(W)-\tr(I_3\otimes I_3)\big)\\
                      &=\frac{1}{54}(3-9)=\frac{-6}{54}=\frac{-1}{9}.
    \end{align*}
\end{proof}

\begin{rmk}
    Note that there are unitaries $B\in\M{3}$  such that $\tr(B^*B^{\T})=-1$. For example, 

 $$\frac{1}{2}\Matrix{0 & 1-i & -1-i\\ -1+i & -i & 1\\ 1+i & 1 & i}, \qquad \Matrix{1&0&0\\0&0&-1\\0&1&0}, \qquad \Matrix{0&-1&0\\1&0&0\\0&0&1}.$$
    
\end{rmk}

\begin{eg}\label{eg-semigrp-not-mixed unitary}
    Let $\mcl{L}$ be a linear map on $\M{3}$  given by
    \begin{align*}
        \mcl{L}(X)=\frac{1}{2}B(\tr(X)I_3-X^{\T})B^{*}-X, ~~~~\forall X\in\M{3},
    \end{align*}
    where $B\in\M{3}$ be any unitary matrix such that $\tr(B^*B^{\T})=-1$. Now, consider the semigroup whose infinitesimal generator is $\mcl{L}$. From Theorem \ref{thm-cp-semigroup}, it follows that $\{e^{t\mcl{L}}\}_{t\geq 0}$ is a dynamical semigroup of unital quantum channels. Take $\Gamma=\frac{1}{2}(\mathrm{Ad}_{B^*}\circ\T+\delta_3)$, then from Lemma \ref{lem-example-not mixed unitary semigroup},  $\Gamma\in\msc{D}(3)$ and $\ip{\Gamma, \mcl{L}}=-\frac{1}{9}<0$. Therefore, by Theorem \ref{thm-schonberg-mixed unitary}, the semigroup $e^{t\mcl{L}}$ is not mixed unitary for some $t>0$.
\end{eg}

\begin{eg}\cite{KuMa87}\label{eg-semigp-not-mu-kumm}
    Let $A\in\M{d}$ be a normal matrix whose spectrum is neither contained in a circle nor on a straight line. Take the linear map $\mcl{L}$ on $\M{d}$
    \begin{align*}
        \mcl{L}(X)=A^*XA-\frac{1}{2}(A^*AX+XA^*A), ~~~\forall X\in\M{d}.
    \end{align*}
    Then the semigroup of unital quantum channel $e^{t\mcl{L}}$ is not mixed unitary for some $t>0$.  
\end{eg}

We have shown in the above Theorem \ref{thm-semigroup-not mu-intially}, that there exists $t_0$ such that $e^{t\mcl{L}}$ is not mixed unitary for any $t\in (0,t_0)$.  We now give an example of a semigroup for which we determine $t_0$ such that $e^{t\mcl{L}}$ is not mixed unitary for any $t\in (0,t_0)$.

\begin{eg}
Let $\mcl{L}$ be the linear map on $\M{3}$ as given in the Lemma \ref{lem-example-not mixed unitary semigroup}. That is,
    \begin{align*}
        \mcl{L}(X)=\frac{1}{2}B(\tr(X)I_3-X^{\T})B^{*}-X, ~~~~\forall X\in\M{3},
    \end{align*} where $B=\Matrix{1&0&0\\0&0&-1\\0&1&0}$. Consider the semigroup $e^{t\mcl{L}}$. According to Theorem \ref{thm-cp-semigroup}, this family constitutes a dynamical semigroup of unital quantum channels. Let $t_0\approx 1.4034$ be the positive root of the function $e^{t}-e^{-\frac{t}{2}}-3\sinh(\frac{t}{2})-2\sin(\frac{t}{2})$.\\
    \textbf{Claim:} $e^{t\mcl{L}}$ is not mixed unitary, for all $0<t<t_0$.\\
    \textbf{Proof of the claim :}
    Write $\mcl{L}=\mcl{L}_1+\mcl{L}_2+\mcl{L}_3$, where $\mcl{L}_j$ be the linear maps on $\M{3}$ given by
 \begin{align*}
     \mcl{L}_1&:=-\frac{1}{2}\mathrm{Ad}_{B^*}\circ\T, \qquad
     \mcl{L}_2:=\frac{3}{2}\delta_3,\qquad
     \mcl{L}_3:=-\id.
 \end{align*}
 Then  $\mcl{L}_j$'s commute with each other, and hence $e^{t\mcl{L}}=e^{t\mcl{L}_1}\circ e^{t\mcl{L}_2}\circ e^{t\mcl{L}_3}$. Consider,
 \begin{align*}
  e^{t\mcl{L}_3}&=\lim_{n}(\id+\sum_{k=1}^n\frac{(-t)^k}{k!}\id)=\lim_{n}(1+\sum_{k=1}^n\frac{(-t)^k}{k!})\id=e^{-t}\id,\\
  e^{t\mcl{L}_2}&=\lim_{n}(\id+\sum_{k=1}^n\frac{(3t)^k}{2^k k!}\delta_3^k)
               =\lim_{n}(\id+\sum_{k=1}^n\frac{(3t)^k}{2^k k!}\delta_3)
               =\id+(e^{\frac{3}{2}t}-1)\delta_3.
 \end{align*}
 Now to compute $e^{t\mcl{L}_1}$, first note that $\mathrm{Ad}_{B^{\T}}\circ\T=\mathrm{Ad}_{B^*}\circ\T=\T\circ\mathrm{Ad}_{B^{\T}}$, $B^2=(B^*)^2=(B^{\T})^2, B=(B^*)^3=(B^{\T})^3$ and $BB^{\T}=I_3=(B^{\T})^4$. This gives us,
 \begin{align*}
     (\mathrm{Ad}_{B^*}\circ \T)^n&=\begin{cases*}
                            \id & \quad \text{if} n=0 mode 4\\
                            \mathrm{Ad}_{B^{\T}}\circ\T &\quad \text{if} n=1 mode 4\\
                            \mathrm{Ad}_{(B^{\T})^2} &\quad \text{if} n=2 mode 4\\
                            \mathrm{Ad}_{B}\circ\T &\quad \text{if} n=3 mode 4.
                          \end{cases*}
 \end{align*}
 Therefore,
   \begin{align*}
    e^{t\mcl{L}_1}&=\lim_{n}(\id+\sum_{k=1}^n\frac{(-t)^k}{2^k k!}(\mathrm{Ad}_B\circ\T)^k)\\
            &=\id+\sum_{k=0 \text{mode} 4}\frac{(-t)^k}{2^k k!}\id+\sum_{k=1 \text{mode} 4}\frac{(-t)^k}{2^k k!}\mathrm{Ad}_{B^{\T}}\circ\T+\sum_{k=2 \text{mode} 4}\frac{(-t)^k}{2^k k!}\mathrm{Ad}_{(B^{\T})^2}+\sum_{k=3 \text{mode} 4}\frac{(-t)^k}{2^k k!}\mathrm{Ad}_{B}\circ\T\\
            &=a_t\id+b_t\mathrm{Ad}_{B^{\T}}\circ \T+c_t\mathrm{Ad}_{(B^{\T})^2}+d_t\mathrm{Ad}_{B}\circ\T, 
   \end{align*}
   and consequently,
   \begin{align*}
       e^{t\mcl{L}}=(1-e^{-\frac{3}{2}t})\delta_3+e^{-t}(a_t\id+b_t\mathrm{Ad}_{B^{\T}}\circ \T+c_t\mathrm{Ad}_{(B^{\T})^2}+d_t\mathrm{Ad}_{B}\circ\T),
   \end{align*}
   where
   \begin{align*}
       a_t=\sum_{n=0}^{\infty}\frac{t^{4n}}{2^{4n}(4n)!}, \quad
       b_t=\sum_{n=0}^{\infty}\frac{(-t)^{4n+1}}{2^{4n+1}(4n+1)!},\quad
       c_t=\sum_{n=0}^{\infty}\frac{t^{4n+2}}{2^{4n+2}(4n+2)!},\quad
       d_t=\sum_{n=0}^{\infty}\frac{(-t)^{4n+3}}{2^{4n+3}(4n+3)!}.
   \end{align*}

Let $\Gamma$ be the map given in the Lemma \ref{lem-example-not mixed unitary semigroup}. That is, $ \Gamma=\frac{1}{2}(\mathrm{Ad}_{B^*}\circ\T+\delta_3)$, then we have 
\begin{align*}
    \ip{\mathrm{Ad}_{(B^{\T})^2},\Gamma}&=\frac{1}{2}\ip{\id,\mathrm{Ad}_{B^2}\circ (\mathrm{Ad}_{B^*}\circ\T+\delta_3)}
                                               =\frac{1}{2}\ip{\id,\mathrm{Ad}_{B}\circ\T+\delta_3)}
                                               =\frac{1}{2}(\ip{\id,\mathrm{Ad}_{B}\circ\T}+\ip{\id,\delta_3})\\
                                                &=\frac{1}{2}\big(\tr(\msc{J}(\id)\msc{J}(\mathrm{Ad}_{B} \circ\T))+\tr(\msc{J}(\id)\msc{J}(\delta_3))\\
                                 &=\frac{1}{6}\big(\ip{\Omega_3,(I_3\otimes B^*)F(I_3\otimes B)\Omega_3}+\frac{1}{3}\ip{\Omega_3,\Omega_3}\big) \qquad(\text{here }F=\sum_{i,j=1}^3E_{ij}\otimes E_{ji})\\
                                    &=\frac{1}{6}\big(\frac{\tr(\ol{B}B)}{3}+\frac{1}{3}\big)=\frac{1}{6}\big(\frac{-1}{3}+\frac{1}{3}\big)=0, 
\end{align*}

\begin{align*}
      \ip{\mathrm{Ad}_{B^{\T}}\circ\T,\Gamma}=\ip{\T,\mathrm{Ad}_{B}\circ\Gamma}&=\frac{1}{2}\ip{\T,\mathrm{Ad}_{B}\circ (\mathrm{Ad}_{B^*}\circ\T+\delta_3)}\\
                                     &=\frac{1}{2}\big(\ip{\id,\id+\delta_3}\big)
                                     =\frac{1}{2}\big(\abs{\ip{\Omega_3,\Omega_3}}^2+\frac{1}{9}\ip{\Omega_3,\Omega_3}\big)
                                     =\frac{1}{2}(1+\frac{1}{9})
                                     =\frac{10}{18},
\end{align*}

and 
\begin{align*}
     \ip{\mathrm{Ad}_{B}\circ\T,\Gamma}=&\ip{\T,\mathrm{Ad}_{B^*}\circ\Gamma}=\frac{1}{2}\ip{\T,\mathrm{Ad}_{B^*}\circ (\mathrm{Ad}_{B^*}\circ\T+\delta_3)}\\
                                     &=\frac{1}{2}\big(\ip{\T,\mathrm{Ad}_{(B^2)^*}\circ\T+\delta_3}\big)
                                     =\frac{1}{2}\big(\ip{\id,\mathrm{Ad}_{(B^2)^*}}+\ip{\id,\delta_3}\big)\\
                                     &=\frac{1}{2}\big(\abs{\ip{(I_3\otimes B^2)\Omega_3,\Omega_3}}^2+\frac{1}{9}\ip{\Omega_3,\Omega_3}\big)
                                     =\frac{1}{2}\big(\frac{1}{9}\abs{\tr(\ol{B}^2)}^2+\frac{1}{9}\big)
                                     =\frac{1}{9}.
\end{align*}

Therefore, for all $t\in(0,t_0)$ we get,
\begin{align*}
    \ip{e^{t\mcl{L}},\Gamma}&=(1-e^{-\frac{3}{2}t})\ip{\delta_3,\Gamma}+e^{-t}\big\{a_t\ip{\id,\Gamma}+b_t\ip{\mathrm{Ad}_{B^{\T}}\circ\T,\Gamma}+c_t\ip{\mathrm{Ad}_{(B^{\T})^2},\Gamma}+d_t\ip{\mathrm{Ad}_{B}\circ\T,\Gamma} \big\}\\
                 &=(1-e^{-\frac{3}{2}t})\frac{1}{9}+e^{-t}(\frac{10}{18}b_t+\frac{1}{9}d_t)\qquad(\because \ip{\id,\Gamma}=0)\\
                 &=\frac{e^{-t}}{9}\big\{e^{t}-e^{-\frac{t}{2}}+5b_t+d_t\big\}\\
                 &=\frac{e^{-t}}{9}\big\{e^{t}-e^{-\frac{t}{2}}+4b_t-\sinh(\frac{t}{2})\big\} \qquad (\because b_t+d_t= \sinh(\frac{-t}{2}))\\
                  &=\frac{e^{-t}}{9}\big\{e^{t}-e^{-\frac{t}{2}}-\frac{4}{2}(\sinh(\frac{t}{2})+\sin(\frac{t}{2}))-\sinh(\frac{t}{2})\big\} \qquad (\because 2b_t=-\sinh(\frac{t}{2})-\sin(\frac{t}{2}))\\
                  &=\frac{e^{-t}}{9}\big\{e^{t}-e^{-\frac{t}{2}}-3\sinh(\frac{t}{2})-2\sin(\frac{t}{2})\big\}
                  <0
\end{align*}
Now, from Lemma \ref{lem-interior-dual}, it follows that  $e^{t\mcl{L}}$ is not mixed unitary for all $0<t<t_0$.
\end{eg}


\section{Semigroups of mixed Weyl-unitary channels and Schur channels}

In this Section we study some special classes of channels which are of interest in quantum infomration theory. A mixed unitary channel is a convex combination of maps of the form $\mathrm{Ad}_{U}$, where $U$ can be arbitrary unitary. Let $\msc{S}$ be a subset of the group of all $d\times d$ unitaries $\mbb{U}(d)$. One can ask, What are all the channels that can be written as convex combinations of maps $\mathrm{Ad}_{U}$, where $U\in\msc{S}$? In this context we recall the notion of Weyl unitaries.

Let $\mathbb {Z}_d=\{0, 1, \ldots , d-1\}$ be the cyclic group of order $d$, and let $\xi =e^{\frac{2\pi i}{d}}$ be the primitive $d$-th root of unity.  We consider the Hilbert space $\mathbb{C}^d$ with orthonormal basis indexed by $\mathbb{Z}_d$: $\{e_j: 0\leq j \leq (d-1)\}$. For every $a,b\in\mbb{Z}_d$, define the unitary operators, $U,V$, and $W_{a, b}$ on  $\mbb{C}^d$ as follows:
\begin{align*}
    U(e_j)&=e_{j+1};\\
    V(e_j)&=\xi^je_j;\\
    W_{a,b}(e_j)&=U^aV^b(e_j),  \qquad  ~\forall j,a,b\in\mbb{Z}_d.
\end{align*}
The operators $W_{a, b}$ are known as Weyl unitaries. The usefulness of Weyl unitaries in quantum information theory is well-known (See \cite{Wat18}). They have the following key properties:

\begin{enumerate}[label=(\roman*)]
    \item (Adjoint) $W_{a,b}^*=\xi^{ab}W_{-a,-b}$, for all $a,b\in\mbb{Z}_d$;
    \item (Commutation relations) $W_{a,b}W_{a',b'}=\xi^{ba'-ab'}W_{a',b'}W_{a,b}$, for all $a,b,a',b'\in\mbb{Z}_d$;
    \item (Orthogonality) $\ip{W_{a,b},W_{a',b'}}=\tr(W_{a,b}^*W_{a',b'})=\begin{cases*}
                                 d & \qquad \text{if} (a,b)=(a',b')\\
                                 0 & \qquad\text{otherwise}
                                 \end{cases*}$.
                                 
\end{enumerate}

The orthogonality condition tells us that, the set $\{\frac{1}{\sqrt{d}}W_{a,b}: a,b\in\mbb{Z}_d\}$ is an orthonormal basis for $\M{d}$ with respect to the Hilbert-Schmidt inner product. For any $a,b,a',b'\in\mbb{Z}_d$,  define $\phi_{(a,b),(a',b')}$ on $\M{d}$ as
\begin{equation*}
    \phi_{(a,b),(a',b')}(X)=W_{a,b}XW_{a',b'}^* \qquad ~\forall X\in\M{d}.
\end{equation*}
Observe that the family $\{\phi_{(a,b),(a',b')}:a,b,a',b'\in\mbb{Z}_d\}$ forms a basis for $\B{\M{d}}$. Therefore, any linear map $\Phi$ on $\M{d}$ is written as $\Phi=\sum_{a,b,a',b'\in\mbb{Z}_d}\lambda_{(a,b),(a',b')}  \phi_{(a,b),(a',b')}$. We let $\lambda_{(a,b),(a,b)}=\lambda_{a,b}$ and $\phi_{(a,b),(a,b)}=\phi_{a,b}$. 

\begin{defn}
 A linear map $\Phi:\M{d}\to\M{d}$ is said to be
\begin{enumerate}[label=(\roman*)]
    \item \it{mixed Weyl-unitary} if there exist positive scalars $\lambda_1, \lambda_2,\ldots,\lambda_l$ add up to $1$ and Weyl-unitaries $U_1, U_2,\ldots, U_l$ in $\M{d}$ such that $\Phi(X)=\sum_{i=1}^{l}\lambda_i U_iX U_i^*$ for all $X\in \M{d}$.
    \item \it{Weyl-covariant} if $\Phi(W_{a,b}XW_{a,b}^*)=W_{a,b}\Phi(X)W_{a,b}^*$ for all $X\in \M{d}$ and $a,b\in \mbb{Z}_d$.
\end{enumerate}
   
\end{defn}

\begin{prop}\cite{Wat18}\label{prop-mix-weyl-covari}
Let $\Phi$ be a unital quantum channel on $\M{d}$. Then the following are equivalent:
\begin{enumerate}[label=(\roman*)]
    \item $\Phi$ is a mixed Weyl-unitary;
    \item $\Phi$ is Weyl-covariant;
    \item $\Phi=\sum_{a,b\in\mbb{Z}_d}\lambda_{a,b}  \phi_{a,b}$ with $\lambda_{a,b}\geq 0$ such that $\sum_{a,b\in\mbb{Z}_d}\lambda_{a,b}=1$.
\end{enumerate}
\end{prop}

\begin{lem}
    The following holds: 
    \begin{enumerate}[label=(\roman*)]
         \item The set of all mixed Weyl-unitaries is a compact, convex set of the real vector space $\msc{H}(d)$.
         \item The completely depolarizing channel $\delta_d$ is mixed Weyl-unitary.
         \item For a unital channel $\Phi$ on $\M{d}$ we have: $\Phi$ is mixed Weyl-unitary if and only if $\ip{\Gamma,\Phi}\geq 0$ for all unital and trace-preserving  $\Gamma\in\msc{H}(d)$ with $\ip{\Gamma,\mathrm{Ad}_U}\geq 0$ for all Weyl-unitaries $U\in\mbb{U}(d)$.
    \end{enumerate}
\end{lem}
\begin{proof}
    $(i)$ It is easy to see that the set of all Weyl-covariant channels is a closed and convex set of the real vector space $\msc{H}(d)$. Now, from Proposition \ref{prop-mix-weyl-covari}, the set of mixed Weyl-unitaries is equal to the set of unital Weyl-covariant channels. Thus, the set of mixed Weyl-unitaries is a closed (hence compact) convex set.\\
    $(ii)$ Since $\{\frac{1}{\sqrt{d}}W_{a,b}: a,b\in\mbb{Z}_d\}$ is an orthonormal basis of $\M{d}$ with respect to the Hilbert-Schmidt inner product, we have 
    \begin{equation}\label{eq-choi-delta}
        I_d\otimes I_d=\sum_{a,b\in\mbb{Z}_d}\frac{1}{d}\ranko{w_{a,b}}{w_{a,b}},
    \end{equation}
    where $w_{a,b}:=\sum_{j\in\mbb{Z}_d}e_j\otimes W_{a,b}^*e_j$. Observe that $\frac{1}{d}\ranko{w_{a,b}}{w_{a,b}}=\msc{J}(\mathrm{Ad}_{W_{a,b}})$ and $\frac{1}{d^2}I_d\otimes I_d=\msc{J}(\delta_d)$. Therefore, by applying $\msc{J}^{-1}$ to the equation \eqref{eq-choi-delta}, we get $\delta_d=\sum_{a,b\in\mbb{Z}_d}\frac{1}{d^2}\mathrm{Ad}_{W_{a,b}}.$\\
    $(ii)$ Proof follows similar to the proof of \ref{lem-interior-dual} (iii).
\end{proof}

\begin{thm}\label{thm-mixed-Weyl}
Let $e^{t\mcl{L}}$ be a dynamical semigroup of unital quantum channels. Then the following are equivalent:
\begin{enumerate}[label=(\roman*)]
    \item $e^{t\mcl{L}}$ is mixed Weyl-unitary for every $t\geq 0$;
    \item $\mcl{L}$ is a non-negative linear combination of maps of the form $\mathrm{Ad}_W-\id$, where $W$ is a Weyl unitary;
    \item $\mcl{L}$ is a Weyl-covariant linear map.
\end{enumerate}
\end{thm}
 \begin{proof}
 $(i)\Rightarrow (ii)$
 Suppose $e^{t\mcl{L}}$ is mixed Weyl-unitary channel for every $t\geq 0$.
 Then for each $t\geq 0$, there exist finitely many Weyl unitaries $\{W_j(t)\}_{j=1}^{n(t)}$ and scalars $\{\lambda_j(t)\}_{j=1}^{n(t)}\subseteq[0,1]$ such that $\sum_{j=1}^{n(t)}\lambda_j(t)=1$ and $e^{t\mcl{L}}=\sum_{j=1}^{n(t)}\lambda_j(t)\mathrm{Ad}_{W_j(t)}$. Consequently,
\begin{align*}
    \frac{e^{t\mcl{L}}-\id}{t}=\sum_{j=1}^{n(t)}\lambda_j(t)(\mathrm{Ad}_{W_j(t)}-\id)\in \text{cone}\{\mathrm{Ad}_W-\id:W\text{ is Weyl unitary}\},
\end{align*} 
Since there are only finitely many Weyl unitaries, so the generating set is finite and hence $\text{cone}\{\mathrm{Ad}_W-\id:W\text{ is Weyl unitary}\}$ is closed. Thus, the map $\mcl{L}=\lim_{t}\frac{e^{t\mcl{L}}-\id}{t}$ is in the cone generated by $\{\mathrm{Ad}_W-\id:W \text{ is Weyl unitary}\}$.\\
 $(ii)\Rightarrow (ii)$ Note that $\mathrm{Ad}_{W}$ is Weyl-covariant map for any Weyl-unitary $W$. Consequently, any map 
 $\Psi$ in the cone generated by $\{\mathrm{Ad}_W-\id:W \text{ is Weyl unitary}\}$ is Weyl-covariant. In particular, $\mcl{L}$ is Weyl-covariant.\\  
 $(iii)\Rightarrow (i)$ Assume that $\mcl{L}$ is a Weyl-covariant linear map. This implies that $\mcl{L}^k$ is  Weyl-covariant for every $k\geq 0$. 
Then by power series expansion, $e^{t\mcl{L}}$ is Weyl-covairant for every $t\geq 0$ and by Proposition \ref{prop-mix-weyl-covari} it follows $e^{t\mcl{L}}$ is mixed Weyl-unitary for every $t\geq 0$.
 \end{proof}

Note that Weyl unitaries along with $d$-th roots of unity form a finite subgroup of the group of unitaries. Most results of this Section extend to such a general context.


\begin{defn}
  Let $G$ be a subgroup of $\mbb{U}(d)$.  An unital channel $\Phi$ on $\M{d}$ is said to be \emph{$G$-mixed unitary} if if there exist $\{U_j\}_{j=1}^\ell\subseteq G$ and  $\{\lambda_j\}_{j=1}^{\ell}\subseteq [0,1]$ for some $l\in \mathbb{N}$ such that $\sum_{j=1}^\ell\lambda_j=1$ and
 \begin{align}
     \Phi(X)=\sum_{i=1}^{\ell}\lambda_i U_i^{*}X U_i, ~~~~~~~\forall X\in \M{d}.
 \end{align}
\end{defn}

Note that any $G$-mixed unitary channel is also a mixed unitary channel. Therefore, the map $\Phi$ given in the Example \ref{eg-not-mu} is an example of a non $G$-mixed unitary map on $\M{3}$.

Denote $\msc{MU}_G(d)$ be the set of all $G$-mixed unitary channels on $\M{d}$. Consider $\msc{MU}_G(d)$ as a subset of the real vector space $\msc{H}(d)$.

\begin{lem} For any compact subgroup $G$ of $\mbb{U}(d)$, we have the following;
    \begin{enumerate}[label=(\roman*)]
        \item  $\msc{MU}_G(d)$ is a compact convex set and is closed under compositions;
        \item $\msc{MU}_G(d)^{\circ}=\{\Gamma\in\msc{H}(d): \ip{\Gamma,\mathrm{Ad}_U}\geq 0~~\forall  U\in G\}$;
        \item  Suppose $\Phi\in\msc{H}(d)$ is unital and trace-preserving. Then $\Phi\in\msc{MU}_G(d)$ if and only if $\ip{\Gamma,\Phi}\geq 0$ for all unital and trace-preserving  $\Gamma\in\msc{MU}_G(d)^{\circ}$;
         \item The interior of the set $\msc{MU}_G(d)$ is empty.
    \end{enumerate}
\end{lem}    

\begin{proof}
    The proof is similar to that of Lemma \ref{lem-interior-dual}.
\end{proof}

Denote $\msc{D}_G(d):=\{\Gamma\in\msc{W}(d): \ip{\id,\Gamma}=0, \ip{\mathrm{Ad}_U,\Gamma}\geq 0,~~\forall U\in G\}$.
\begin{lem}
Suppose $\mcl{L}\in\msc{H}(d)$ with $\ip{\Gamma, \mcl{L}}\geq 0$ for all $\Gamma\in\msc{D}_G(d)$. Then, for every $r>0$ there exists $\eta >0$ such that if  $\Gamma\in\mcl{C}(d)^{\circ}\cap\msc{W}_1(d)$ satisfies $\ip{\Gamma,\id}<\eta $ then $\ip{\Gamma,\mcl{L}}\geq -r$.
\end{lem}
\begin{proof}
Similar to Lemma \ref{lem-mixed unitary-semigroup}. 
\end{proof}

\begin{thm}\label{thm-semi-G-MU}
     Let $\mcl{L}\in\msc{H}(d)$ be such that $\mcl{L}(I_d)=0$ and $\tr(\mcl{L}(X))=0, \forall X\in\M{d}$. Then, the following are equivalent:
     \begin{enumerate}[label=(\roman*)]
         \item $e^{t\mcl{L}}$ is $G$-mixed unitary for all $t\geq 0$;
         \item The generator $\mcl{L}$ is in the closed cone generated by $\{\mathrm{Ad}_U-\id:U\in G\}$; 
         \item $\ip{\Gamma,\mcl{L}}\geq 0$ for all $\Gamma\in\msc{D}_G(d)$.
     \end{enumerate}
 \end{thm}
 \begin{proof}
     Follows in the same lines of Theorem \ref{thm-schonberg-mixed unitary}.
 \end{proof}
Note that if the group is finite, in part (ii) of this theorem it suffices to take the non-negative linear combinations of maps of the form $\{\mathrm{Ad}_U-\id:U\in G\}$ as there is no need to take the closure.

 Now we consider another important class of channels called Schur channels.
Recall that a CP map $\Phi:\M{d}\to\M{d}$ is said to be a \it{Schur channel} if there exists a $C\in\M{d}^+$ with unit diagonal entries such that
$$\Phi(X)=\mbb{S}_C(X):=C\circ X, $$ 
where $C\circ X$ is the Hadamard (or Schur)  product of $C$ and $X$. A positive matrix $C\in\M{d}^+$  with unit diagonal entries is referred as a \it{correlation matrix}. It is straightforward to verify that, a CP map $\Phi$ is a Schur channel if and only if it admits a Kraus representation with all Kraus operators being diagonal. Since the linear span of Kraus operators of a CP map is independent of the representation it follows that in
any Choi-Kraus representation of a Schur channel all Kraus operators are diagonal.

\begin{defn} A correlation matrix is said to be a {\em mixed rank-one\/} matrix if it is a convex combination of rank one correlation matrices.
\end{defn}

Observe that a rank one correlation matrix is of the form $B=[e^{i(b_j-b_k)}]_{1\leq j,k\leq d}$ for some real numbers $b_1, b_2, \ldots , b_d.$
In such a case, $\mbb{S}_B=\mathrm{Ad}_U$, where $U$ is the diagonal matrix with entries $e^{-ib_1}, e^{-ib_2}, \ldots , e^{-ib_d}.$
Moreover, it is well-known that  a Schur channel $\mbb{S}_C$ is mixed unitary if and only if $C$ is a mixed rank-one correlation matrix. (See the proof of Lemma 2.4 in \cite{BPS93} or Theorem 5 \cite{CoRa13}). Schur channels which are not of this form known to exist for $d\geq 4.$

The correlation matrices are of interest in quantum theory due to their connection with Bell's inequality and related concepts (\cite{BGP98}). They are of interest in quantum information theory as they provide very simple and natural examples of unital quantum channnels (\cite{Wat18}). 
There exists extensive literature on correlation matrices and associated quantum channels. Nevertheless the following result on eventual mixed rank-one property seems to be new.

\begin{thm}\label{thm-schur-even}
Let $C=[c_{jk}]_{1\leq j,k\leq d}$ be a correlation matrix. Then $\{ [c_{jk}^n]:n\in {\mathbb{N}}\}$ is eventually mixed rank-one, that is, there exists $K\in \mathbb{N}$ such that $[c_{jk}^n]$ is mixed  rank-one for all $n\geq K.$ Similarly, if $A=[a_{jk}]_{1\leq j,k\leq d}$ is a matrix such
that $[e^{ta_{jk}}]:t\geq 0\}$ are correlation matrices, then there exists $T\geq 0$ such that $[e^{ta_{jk}}]$ is mixed rank-one for all $t\geq T.$
\end{thm}
\begin{proof} The first part is an immediate consequence of our main Theorem \ref{thm-even-mu-discrete}, by considering the CP channel $\mbb{S}_C$ along with the observation made above that if a Schur channel of a matrix is mixed unitary then the matrix is a mixed rank-one matrix. The second part follows from
the one parameter semigroup version of the main theorem proved in  Theorem \ref{thm-semigp-eve-MU}.
\end{proof}

Now we ask when does a semigroup of Schur channels remain mixed unitary for all times? Since a CP map is a Schur channel if and only if it admits a Kraus representation consisting of diagonal operators, this question is equivalent to determining when a semigroup of unital channels is $G$-mixed unitary for all times, where $G$ denotes the group of diagonal unitary matrices.
\begin{thm}\label{thm-schur-gene}
     Let $A=[a_{jk}]\in\M{d}.$  Consider the semigroup of linear maps  $\Phi_t=\mbb{S}_{C(t)}$, where $C(t)=[e^{ta_{jk}}]\in\M{d}$. Then the following are equivalent:
\begin{enumerate}
\item         $\Phi_t$ is mixed unitary for all $t\geq 0$;
\item $C(t)$ is mixed rank-one for all $t\geq 0$;
\item The matrix $A$ is in the closed cone generated by  $\{[e^{i(b_j-b_k)}-1]_{1\leq j,k\leq d}: b_j\in \mathbb{R},~\forall j\}$;
\item $A$ is a positive linear combination of matrices of the following form: (i) \{$[i(h_j-h_k)]_{1\leq j,k\leq d} :h_j\in \mathbb{R}, ~\forall j\}$;
(ii) $\{ [-(a_j-a_k)^2]_{1\leq j,k\leq d}: a_j\in \mathbb{R},~\forall j\};$ (iii) $\{[e^{i(b_j-b_k)}-1]_{1\leq j,k\leq d}: b_j\in \mathbb{R},~\forall j\}$. \end{enumerate}
 \end{thm}
\begin{proof} We have already seen the equivalence of $(1)$ and $(2).$
    Let $\Phi_t=\mbb{S}_{C(t)}$ be the given semigroup of linear maps. Then its generator $\mcl{L}$ is given by $\mcl{L}=\mbb{S}_{A}$. Let $G$ be the group of diagonal unitary matrices. Observe that, for any $t\geq 0$, $\Phi_t$ is mixed unitary if and only if $\Phi_t$ is $G$-mixed unitary. Therefore, from Theorem \ref{thm-semi-G-MU}, $\Phi_t$ is mixed unitary for all times if and only if $\mcl{L}$ is in the closed convex cone generated by  $\{\mathrm{Ad}_U-\id:U\in G\}$. Since $U$ is diagonal the map $\mathrm{Ad}_U$ is a Schur channel of the form  $\mbb{S}_{[e^{i(b_j-b_k)}]}$. 
    As the map $\M{d}\ni B\mapsto \mbb{S}_B$ is linear, $(3)$ follows. Now from \cite{KuMa87} or Theorem \ref{thm-mu-semi-gene}, the generator of the semigroup has the form
    \begin{equation*}
            \mbb{S}_C(X)=i[H,X]+\frac{1}{2}\sum_{j=1}^{n}\lambda_j(2L_j^*XL_j-L_j^*L_jX-XL_j^*L_j),
        \end{equation*}  with self-adjoint $H$ and either self-adjoint or unitary $L_j's$ and $\lambda _j\geq 0.$ Since as Schur property forces $H$, $L_j$'s to be diagonal we get $(4)$. The implication $(4)$ implies $(1)$ is clear from Theorem \ref{thm-mu-semi-gene}.
\end{proof}

\section*{Acknowledgements}
Bhat gratefully acknowledges funding from ANRF (India) through J C Bose Fellowship No. JBR/2021/000024. Devendra acknowledges funding from the Indian Institute of Technology Bombay through the Institute Postdoctoral Fellowship and also thanks the Indian Statistical Institute Bangalore for its kind hospitality during his visits. 

\subsection*{Funding}Bhat gratefully acknowledges funding from ANRF (India) through J C Bose Fellowship No. JBR/2021/000024. 
\section*{Data availability}
The manuscript has no associated data.
\section*{Declarations}
\subsection*{Conflict of interest} The authors have no conflict of interest.
\subsection*{Ethical statement} This article does not contain any studies involving human
participants performed by any of the authors.
\subsection*{Informed consent} It is not relevant for this study.
\printbibliography
\end{document}